\documentclass[10pt,leqno]{amsart}
\usepackage{graphicx}
\usepackage{indentfirst,csquotes}

\usepackage{amssymb,amsthm,amsmath}
\numberwithin{equation}{section}
\usepackage{xcolor,paralist,hyperref,titlesec,fancyhdr,etoolbox}
\newtheorem{thm}{Theorem}[section]

\newtheorem{lem}{Lemma}[section]

\newtheorem{rem}{Remark}[section]

\titleformat{\section}[hang]{\normalfont\large\bfseries}{\thesection. }{0pt}{}
\titlespacing*{\section}{0pt}{2ex}{2ex}

\titleformat{\subsection}[hang]{\normalsize\bfseries}{\thesubsection. }{0pt}{}
\titlespacing*{\section}{0pt}{2ex}{2ex}

\hypersetup{ colorlinks=true, linkcolor=black, filecolor=black, urlcolor=black }

\usepackage{lipsum}

\begin{document}
	\title{Global solutions of compressible Navier-Stokes equations with small viscosity} 

\author{Lv Cai}
\address{ Department of Mathematics, Shanghai University, Shanghai 200444, China}
\email{cailv@shu.edu.cn}

\author{Ning-An Lai}
\address{Department of Mathematics, Shaoxing University, Shaoxing 312000, China}
\email{ninganlai@usx.edu.cn(N.-A.Lai)}



\author{Zexian Zhang}
\address{School of Mathematical Sciences, Fudan University, Shanghai 200433, China}
\email{23110840019@m.fudan.edu.cn(Z. X. Zhang)}
\author{Yi Zhou}
\address{School of Mathematical Sciences, Fudan University, Shanghai 200433, China}
\email{yizhou@fudan.edu.cn(Y. Zhou)}

	\date{\today}
	\maketitle
	
	\let\thefootnote\relax


\begin{abstract}

In this paper, we study the Cauchy problem for the compressible Navier-Stokes system in $\mathbb{R}^3$. Suppose that the viscosity coefficients satisfy $0<\max\{\mu, \nu=\lambda+2\mu\}<1$, and set $\varepsilon=\min\{\mu, \nu=\lambda+2\mu\}$. We establish the global existence of classical solutions when the initial perturbations of the density and the curl-free part of the velocity are smaller than $\varepsilon^{\frac12+}$ (up to a logarithmic loss), while the divergence part of the initial velocity is smaller than $\varepsilon$. This improves the classical global existence result of Matsumura-Nishida \cite{MaN80}, which requires all the initial data to be smaller than $\varepsilon (<1)$. We expect that this result is representative of general Shizuta-Kawashima systems arising in physical applications. 
The improvement of the index from $1$ to $\frac12+$ relies on exploiting the hidden Kawashima-type dissipation for the density and controlling the spacetime trace norm of the solution at the scale $\sqrt{\varepsilon}$. These two ingredients are obtained through a weighted trace inequality and a Morawetz-type inequality for the perturbed sound speed and the divergence of the velocity.

\end{abstract}

	\section{Introduction}
In this paper, we consider the Cauchy problem of 3D compressible Navier-tokes equations with initial data close to a constant equilibrium
\begin{equation}\label{eq1}
	\begin{aligned}
		\begin{split}
			\left\lbrace
			\begin{array}{lr}
		\partial_t \rho + \operatorname{div}(\rho u) = 0,~~~~~~~in~(0, \infty)\times \mathbb{R}^3, \\
		\rho \bigl( \partial_t u + (u \cdot \nabla) u \bigr) + \nabla p(\rho) = \mu \Delta u + (\lambda + \mu) \nabla (\operatorname{div} u),~~in~(0, \infty)\times \mathbb{R}^3, \\
		(\rho,u)|_{t=0} = (\rho_0,u_0),~~~~~~~~in~ \mathbb{R}^3,
			\end{array}	
\right.
\end{split}	
\end{aligned}
\end{equation}
where $\rho$ is the density of the fluid, $u$ is the velocity vector, and $p(\rho)$ represents the pressure as a function of $\rho$. The constants $\lambda, \mu$ are defined as two Lam\'{e} coefficients of the fluid, which are assumed to satisfy $\mu > 0, \nu=\lambda + 2\mu > 0$. Such conditions ensure the ellipticity of the
operator $\mu\Delta+(\lambda+\mu)\nabla div$ and are satisfied in the physical cases.

The foundational work for the global existence of classical solution with small initial data for system \eqref{eq1} is due to Matsumura-Nishida \cite{MaN80}, which is the starting point of the global theory with small perturbation for the compressible Navier-Stokes system. Also, the Matsumura-Nishida energy method was developed, we refer to the series of works of the same authors for the corresponding initial boundary value problem \cite{MaN81, MaN83, MaN831}. This kind of global theory was generalized by Kawashima \cite{Kaw83} (the seminal work on symmetric hyperbolic
partially diffusive systems), Shizuta-Kawashima \cite{ShK85} and Kawashima-Shizuta \cite{KaS88} systematically to the general symmetric hyperbolic-parabolic systems. These works come up with a well-known Kawashima-Shizuta condition, which exhibits a sufficient condition for global well-posedness and essentially says that propagation of hyperbolicity and partial dissipation will lead to dissipation for the whole system. It is interesting to see that the Kawashima-Shizuta condition works for almost all the dissipative or diffusive models, like compressible Navier-Stokes equations, compressible magnetohydrodynamics, compressible Euler with damping and so on. The requirements in the thesis \cite{Kaw83} was justified by Serre \cite{Ser10}, and the same author \cite{Ser101} improved the results in \cite{KaS88}.

Noting that if $(\rho, u)$ is a solution to the system \eqref{eq1} with initial data $(\rho_0, u_0)$, then the rescaled pair
\[
\left(\rho(l^2t, lx), lu(l^2t, lx)\right)
\]
is also a solution for all $l>0$, with the rescaled initial data $(\rho_0(lx), lu_0(lx))$, if we replace the pressure function by $l^2 p$. This would lead the concepts of scaling invariance
and critical spaces for system \eqref{eq1} in $\mathbb{R}^d$: system \eqref{eq1} is invariant under the above scaling transformation, and a function space $F\subset \mathcal{S}'\times (\mathcal{S}')^d$ is critical if its associated norm is invariant under the transformation
\[
(\rho, u)\rightarrow \left(\rho(l\cdot), lu(l\cdot)\right).
\]
Obviously, the Besov space $\dot{B}^{\frac dp}_{p, 1}\times \left(\dot{B}^{\frac dp-1}_{p, 1}\right)^d$ is a critical space in this sense. Danchin \cite{Dan00} first established the global existence result for \eqref{eq1} with small initial perturbations $(\rho_0-\overline{\rho}, u_0)$ in the critical Besov space $\dot{B}^{\frac d2}_{2, 1}\times \left(\dot{B}^{\frac d2-1}_{2, 1}\right)^d$ for $d\ge 2$. The same author restudied this problem in \cite{Dan16}, with the high-frequency part of the norm
generalized to $p$ close to $2$. For the global existence with small initial data in Besov and critical weighted Besov spaces in $\mathbb{R}^d (d\ge 3)$, we refer to \cite{Oha23}.

Let $\varepsilon$ be the small viscosity coefficient in the classical Kawashima system, or in MHD, or in the incompressible viscoelasticity, or the mean free path in the Boltzmann equation. In our case for the compressible Navier-Stokes system \eqref{eq1}, we set $\varepsilon=\min\{\mu, \nu=\lambda+2\mu\}$, and assume that $0<\max\{\mu, \nu=\lambda+2\mu\}<1$.  If the size of the initial perturbations are assumed to be $\varepsilon$ to the power $\theta$, thus, $\varepsilon^\theta$, then we call $\theta$ the index of initial data. In the classical global existence theory with small data, all the existence results mentioned above are of index $1$, that is, the initial perturbation of the velocity $u_0$ and density $\rho_0$ satisfy
\begin{equation} \label{classical}
	\begin{aligned}
		\| u_0, \rho_0 - \bar{\rho} \|_{F} \lesssim  \varepsilon,
	\end{aligned}
\end{equation}
where $\|\cdot\|_{F}$ denotes the norm associated with the functional space $F$. Take the incompressible Navier-Stokes system 
\begin{equation}\label{inNS}
u_t+u\cdot\nabla u+\nabla p-\varepsilon \Delta u=0
\end{equation}
for example, it admits a critical space $F$, the norm of which stays invariant under the scaling transformation 
\[
u_\lambda(t, x)=\lambda u(\lambda^2t, \lambda x).
\]
By employing the classical energy method, we want to absorb/control the nonlinear term by the dissipative term $\varepsilon \Delta u$, and this requires that the (initial) energy should be less than $\varepsilon$.  The most important critical spaces for the system \eqref{inNS} in $\mathbb{R}^3$ are
\[
H^{\frac12}\subset L^3\subset \dot{B}^{-1+\frac 3p}_{p, q}\subset BMO^{-1},
\]
and the global existence results with small initial data $(\|u_0\|_F\lesssim \varepsilon)$ are referred to Fujita-Kato \cite{FuK64}, Kato \cite{Kato84}, Cannone \cite{Can95, Can04} and Koch-Tataru\cite{Koch01} respectively.

A natural question is that whether one can decrease the index $\theta$ to be below $1$, it admits significant importance in the vanishing viscosity limit/
low Mach number limit/hydrodynamic limit theory. Such consideration can be dated back to Klainerman-Majda \cite{KlM81}, since then the relationship between the initial data size and viscosity (small parameter) size attracts more and more attention, see \cite{Fuj24, JaJ14, JaN06, Liu14} and references therein. 
However, it may not always possible to take $\theta$ below $1$, for example, for the 3D incompressible Navier-Stokes equation, this seems no way at least so far. This is because if one could succeed in achieving this goal, then the natural scaling of the Navier-Stokes system will lead to large data global existence, and hence the Millennium Problem for the incompressible Navier-Stokes equations is solved.

Another important and closely related problem is the stability and transition to turbulence of the laminar flows at high Reynolds number $Re$, where the parameter $\varepsilon$ is corresponding to the small viscosity coefficient $\nu=Re^{-1}>0$. Trefethen-Trefethen-Reddy-Driscoll \cite{Tre93} first proposed a question: Given a norm $\|\cdot\|_{X}$, find a $\theta=\theta (X)$ such that
\[
 \begin{aligned}
&\|u_0\|_{X}\le \nu^{\theta}\Rightarrow~~ stability,\\
&\|u_0\|_{X}\gg \nu^{\theta}\Rightarrow~~ instability,\\
 \end{aligned}
\]
where the index $\theta$ is referred to as the transition threshold in the applied literature, and it was conjectured in \cite{Tre93} that $\theta\le 1$. Bedrossian-Germain-Masmoudi \cite{Bed17} studied the stability threshold problem for the $3$-D Couette flow in $\mathbb{T}\times \mathbb{R}\times \mathbb{T}$, and proved that if the initial perturbation $u_0$ is in Sobolev space satisfying
\[
\|u_0\|_{H^{\sigma}}\le \delta \nu^{\frac 32},~~~for~\sigma>\frac 92,
\]
then the solution is global in time. This result was improved by Wei-Zhang \cite{Wei21}, in  which the initial perturbation satisfies
\[
\|u_0\|_{H^2}\le c_0\nu
\]
instead, where $c_0>0$ is some constant and $0<\nu<1$. The enhanced dissipation due to the Couette mixing and the inviscid damping play a crucial role in studying the stability threshold. Unfortunately, it seems hard to improve the index to be below $1$ for the $3$D Couette flow due to the lift-up effect. However, it is possible for the Couette flow for the $3$D
incompressible Navier-Stokes-Coriolis system in the high Reynolds number regime. In a recent work by Li-Sun-Wang-Wei-Zhang \cite{Li25}, they explored the rotation dispersion to suppress the lift-up effect and 
proved that if the initial perturbation satisfies
\[
\|u_0\|_{\widetilde{H}(\mathbb{T}\times \mathbb{D})}\le \epsilon_0\nu^{\theta},
\]
where 
\[
 \begin{aligned}
&\theta>\frac 23,~~\widetilde{H}=H^6\cap W^{3, 1},~~~for~\mathbb{D}=\mathbb{R}^2,\\
&\theta\ge\frac 56,~~\widetilde{H}=H^6,~~~for~\mathbb{D}=\mathbb{R}\times \mathbb{T},\\
 \end{aligned}
\]
then the corresponding solution of the Navier-Stokes-Coriolis system exists globally
in time and remains asymptotically close to the Couette flow. 

 Assuming the viscosity is relatively small, i.e., \( \varepsilon=\min(\mu,\nu) <1\), we are going to improve the classical global existence result for system \eqref{eq1}  by establishing some sharper estimates for the solution if the initial perturbations satisfy
\begin{equation}\label{initial data}
	\begin{aligned}
		\begin{split}
	\left\lbrace
	\begin{array}{lr}
		\left\| \Gamma^{\leq N} (\rho_0 - \bar{\rho}, (u_{0})_{cf})  \right\|_{L_x^2} \lesssim \varepsilon^{\frac{1}{2}} ( \ln \varepsilon^{-1} )^{-1} , \\
\left\| \Gamma^{\leq N} (u_{0})_{df} \right\|_{L_x^2}  \lesssim \varepsilon,
			\end{array}	
\right.
\end{split}	
\end{aligned}
\end{equation}
where $\Gamma = \{ \partial_x,\tilde{\Omega} \} $, $\tilde{\Omega}=\Omega_{ij} I + e_i \otimes e_j - e_j \otimes e_i$, $\Omega_{ij}=x_j\partial_{x_i}-x_i\partial_{x_j}, i, j=1, 2, 3$ denotes the rotational vector field. And
\[u_{cf} = \mathbb{P}_{cf} u = \Delta^{-1} \nabla(\operatorname{div} u), u_{df} = \mathbb{P}_{df} u = -  \Delta^{-1} \nabla \times (\nabla \times u)
\]
represent the curl free part and divergence part of the solution respectively.
We shall utilize the rotational invariance of the system, which is inspired by the vector field method introduced in the 1980s by Klainerman \cite{Kal85} to analyze the
decay properties of the linear wave equation. However, the compressible Navier-Stokes system is no longer Lorentz invariant, hence our method differs from that in \cite{Kal85}. Obviously the index in $\eqref{initial data}_1$ is almost $\frac12$, with only a logarithmic loss. We believe it should be $\frac12$ exactly. Actually, we have a \textbf{general conjecture} which states as: \textit{the index $\theta$ mentioned above can be improved to belwo $1$ for almost all the Kawashima type system, including 
\begin{itemize}
\item  global existence for compressible elasticity with small viscosity near constant equilibrium;\\
\item global existence for incompressible magnetohydrodynamics with small viscosity and zero magnetic diffusivity under a nonzero magnetic background;\\
\item global existence for compressible magnetohydrodynamics under a nonzero magnetic background when both viscosity and magnetic diffusivity are of size $\varepsilon$;\\
\item global existence near global Maxwellian of Boltzman equations with small mean free path;\\
\item global existence for compressible Euler system with damping $\varepsilon \rho u$ near constant equilibrium;\\
\item global existence near constant equilibrium of $1D$ quasilinear hyperbolic system with a small damping coefficient on the inhomogenous term.\\
\end{itemize}
 Moreover, if the general conjecture above holds, what is the optimal index? Is it should be $\frac{1}{2}$?} Noting that $\theta=0$ corresponds to the vanishing viscosity theory, while $\theta=1$ corresponds to the classical global existence result, and the optimality of $\frac12$ is a compromise index. The vanishing viscosity ($\theta=0$) result for incompressible magnetohydrodynamic system can be referred to \cite{Cai18, He18, Wei17}, while the classical global existence ($\theta=1$) results for the
 incompressible magnetohydrodynamic system with small Alfv\'{e}n numbers, partially dissipative
quasilinear hyperbolic systems $1D$, the compressible Euler system with damping and
 the incompressible viscoelastic system are referred to \cite{Jiang26}, \cite{Zhou22}, \cite{Wan01, Sid03} and \cite{Lei08} respectively, and references therein. 

 The proof is based on some key ingredients. The first one is to explore the Kawashima type hidden dissipation for the disturbed sound speed $\tau$, by coupling the equation for the divergence of the velocity $\nabla\cdot u$. This would remedy the dissipation for $\tau$, which is lost when doing energy estimate for the coupled system of $\tau$ and $u$. The second one is to control the spacetime trace norm of the solution with scale $\sqrt{\varepsilon}$, by combining two crucial inequalities: a weighted trace inequality and a Morawetz type inequality.

Our main result is stated as follows.
\begin{thm}
\label{thm:main}
Consider the $3$D compressible Navier--Stokes system \eqref{eq1} around the constant equilibrium state $(\bar{\rho}, 0)=(1,0)$, associated with the state equation $p(\rho) = \frac{\rho^\gamma}{\gamma} (\gamma > 1)$. Assume that $\varepsilon=\min\{\mu, \nu=\lambda+2\mu\}$ and $0<\max\{\mu, \nu=\lambda+2\mu\}<1$.
Suppose the initial perturbations $(\rho_0-1, u_0)$ have compact support and satisfy 
\[
\rho_0 - 1 \in H^{N}(\mathbb{R}^3), \quad u_0 \in H^{N}(\mathbb{R}^3)
\]
for some sufficiently large integer $N \ge 10$, and there exists a constant $\delta > 0$ small enough but independent of $\varepsilon$ such that 
\[
\left\|\Gamma^{\leq N}(\rho_0 - 1, u_{0,cf})\right\|_{L^2_x} \leq \delta \, \varepsilon^{\frac{1}{2}} \left(\ln \varepsilon^{-1}\right)^{-1},
\]
\[
\left\|\Gamma^{\leq N} u_{0,df}\right\|_{L^2_x} \leq \delta \, \varepsilon,
\]
where 
\[
u_{0,cf} = \mathbb{P}_{cf} u_0, \quad u_{0,df} = \mathbb{P}_{df} u_0
\]
represent the curl-free and divergence-free parts of the initial velocity respectively, and $\Gamma = \{\partial_x, \tilde{\Omega}\}$ denotes the collection of spatial and rotational vector fields defined in Section 2.1. Then there exists a unique global-in-time solution $(\rho, u)$ of \eqref{eq1} such that for all $t > 0$,
\begin{equation}\label{ineq1}
\begin{aligned}
& \quad \left\|\Gamma^{\leq N}(\rho - 1, u)\right\|_{L^\infty(0,T;L^2_x)} 
+ \varepsilon^{\frac{1}{2}} \left\|\nabla \Gamma^{\leq N} u\right\|_{L^2(0,T; L^2_x)} 
+ \varepsilon^{\frac{1}{2}}\left \|\nabla \Gamma^{\leq N-1} (\rho-1)\right\|_{L^2(0,T; L^2_x)} \\
& \lesssim   \varepsilon^{\frac{1}{2}} \left(\ln \varepsilon^{-1}\right)^{-1}.
\end{aligned}
\end{equation}
Furthermore, the solution satisfies the weighted spacetime estimate 
\begin{equation}\label{ineq1a}
\begin{aligned}
\left\|\langle r \rangle^{-\frac{1}{2}} \nabla \Gamma^{\leq N_{\mathrm{lo}} } (\rho - 1, u_{cf})\right\|_{L^2(0,T; L^2_x)} 
\lesssim  \varepsilon^{\frac{1}{2}} \left(\ln \varepsilon^{-1}\right)^{-\frac{1}{2}},
\end{aligned}
\end{equation}
where $N_{\mathrm{lo}}$ is defined in \eqref{def of low hi} below.
\end{thm}

\begin{rem}

The possible improvement for the index from $1$ to $\frac12+$ is due to the observation of the wave enhanced dissipation for the density and the divergence of the velocity, see $\eqref{eq tilde tau u}_1$ and $\eqref{eq tilde n m}_1$ respectively, which is lost for the divergence free part of the velocity $u_{df}$.

\end{rem}

\begin{rem}

Noting that from \eqref{ineq1}, it is easy to see that the spacetime $L^2$ norm of the highest order derivative of the density admits different scale size.

\end{rem}

We shall prove the above theorem by a bootstrap argument.
Let $N \ge 10$ be a sufficiently large integer and 
\begin{equation}\label{def of low hi}
N_{\mathrm{lo}}:=\left[\frac{N}{2}\right]+4.
\end{equation}
Throughout this section, we assume that all solutions are smooth and satisfy the a prior bounds below on a time interval $[0,T]$:

\medskip

\noindent \textbf{Bootstrap assumptions.}
Assuming that for some sufficiently large constant $M>1$, the following estimates hold:
\begin{equation}
\begin{aligned}
&\left\|\Gamma^{\le N}(\tau,u)\right\|_{L^\infty(0,T;L^2_x)}
+ \mu^{1/2}\left\|\nabla \Gamma^{\le N} u_{\mathrm{df}}\right\|_{L^2(0,T; L^2_x)}
+ \nu^{1/2}\left\|\nabla \Gamma^{\le N} u_{\mathrm{cf}}\right\|_{L^2(0,T; L^2_x)} \\
&+ \nu^{1/2} M^{-1} \left\|\nabla \Gamma^{\le N-1} \tau\right\|_{L^2(0,T; L^2_x)}
\le M \delta  \varepsilon^{1/2} (\ln \varepsilon^{-1})^{-1},
\end{aligned}
\tag{BA1}
\end{equation}

\begin{equation*}
\left\|\Gamma^{\leq N_{\mathrm{lo}}} u_{\mathrm{df}}\right\|_{L^\infty(0,T;L^2_x)}
+ \mu^{1/2}\left\|\nabla \Gamma^{\leq N_{\mathrm{lo}}} u_{\mathrm{df}}\right\|_{L^2(0,T; L^2_x)}
\le M \delta \varepsilon,
\tag{BA2}
\end{equation*}

\begin{equation*}
\left\|\langle r\rangle^{-1/2} \nabla
\Gamma^{\leq N_{\mathrm{lo}}}(\tau,u_{\mathrm{cf}})\right\|_{L^2(0,T; L^2_x)}
\le M^3 \delta   \varepsilon^{1/2} (\ln \varepsilon^{-1})^{-1/2}.
\tag{BA3}
\end{equation*}

\medskip

\begin{rem}
(BA1) corresponds to the high-order energy estimate, including viscous dissipation for both compressible and incompressible components. (BA2) captures improved dissipation at low order estimate for the divergence-free component, which benefits from parabolic estimate. (BA3) shows the weighted Morawetz-type spacetime control for the compressible variables.

\end{rem}

\begin{rem}\label{rem14}

From (BA1), it is reasonable to assume that $|\tau|\le \frac12$ over $[0, T]$.

\end{rem}

\medskip

We aim to improve the constant $M$ to $M/2$, thereby closing the bootstrap argument.

	\section{Preliminaries}
	
	\subsection{Rotational Symmetry of Compressible Navier-Stokes Equations}

The compressible Navier-Stokes equations possess rotational invariance in
$\mathbb R^3$. The corresponding infinitesimal generators yield the
rotational vector fields, which form a family of commuting vector fields for
the system. This structure allows us to propagate angular regularity and
derive higher-order energy estimates.

For a vector-valued function $w(t,x)$ and a scalar function $f(t,x)$, we
define their transformations under a rotation $Q\in SO(3)$ by
\begin{equation}
\begin{aligned}
w_Q(t,x)&=Q^{T}w(t,Qx),\\
f_Q(t,x)&=f(t,Qx).
\end{aligned}
\end{equation}
Here, $Q(s)$ is the one-parameter subgroup of $SO(3)$ defined by
\[
Q(s)=\exp\big(s(e_i\otimes e_j-e_j\otimes e_i)\big).
\]
For a scalar function $f$ and a vector-valued function $w$, we define the
corresponding infinitesimal rotation operators by
\begin{equation}
\begin{aligned}
\widetilde{\Omega}_{ij}f
&:=\left.\frac{d}{ds}f_{Q(s)}\right|_{s=0},\\
\widetilde{\Omega}_{ij}w
&:=\left.\frac{d}{ds}w_{Q(s)}\right|_{s=0}.
\end{aligned}
\end{equation}
By the definition of the rotation transformation, for scalar functions one
has
\begin{equation}
\widetilde{\Omega}_{ij}f
=
\Omega_{ij}f,
\end{equation}
where
\begin{equation}
\Omega_{ij}=x_j\partial_i-x_i\partial_j
\end{equation}
is the classical rotational vector field. However, for vector-valued
functions, the infinitesimal rotation also acts on the vector components.
More precisely,
\begin{equation}
\widetilde{\Omega}_{ij}w
=
\Omega_{ij}w+
(e_i\otimes e_j-e_j\otimes e_i)w .
\end{equation}

We now rewrite equations \eqref{eq1} in terms of the
perturbed sound speed. Let 
\begin{equation} \label{def of tau}
 p(\rho) = \frac{\rho^\gamma}{\gamma}, \; c(\rho) =  \rho^{\frac{\gamma-1}{2}}, \; \tau = c(\rho)-1. 
\end{equation}
The system can be rewritten as
\begin{equation}\label{eq tau u}
	\begin{aligned}
		\begin{split}
			\left\lbrace
			\begin{array}{lr}
		\partial_t \tau + b(\tau+1) \operatorname{div} u + (u \cdot \nabla) \tau = 0, \\ 
\partial_t u + (u \cdot \nabla) u + b^{-1}(\tau+1) \nabla \tau = \bar{\mu} \Delta u + (\bar{\lambda} + \bar{\mu}) \nabla (\operatorname{div} u),
			\end{array}	
			\right.
		\end{split}	
	\end{aligned}
\end{equation}	
	where
\begin{equation} \label{def of b}   
(\bar{\lambda}, \bar{\mu}) := \rho^{-1} (\lambda,\mu), \ b := \frac{\gamma-1}{2}.
\end{equation}	
We state the following lemma.

	\begin{lem} \label{lem 2.1}
		Let \( (\tau, u) \) be a solution of the system \eqref{eq tau u}. Then, for any rotation $ Q \in SO(3)$, the rotated pair \( (\tau_Q, u_Q) \) is also a solution of system \eqref{eq tau u}.
	\end{lem}

\begin{proof}
It suffices to verify that all differential operators appearing in the equations commute with the rotation transformation. Let
$Q=(q_{ij})$ satisfy
\[
Q^TQ=QQ^T=I.
\]
For the divergence operator, we have
\begin{equation}
\begin{aligned}
\nabla\cdot w_Q
&=\partial_i(q_{ji}w_j(Qx,t))\\
&=q_{ji}q_{ki}\partial_k w_j(Qx,t)\\
&=(\partial_j w_j)(Qx,t)
=(\nabla\cdot w)^Q .
\end{aligned}
\end{equation}
Similarly, for scalar functions,
\begin{equation}
\begin{aligned}
(w_Q\cdot\nabla)f_Q
&=q_{ji}w_j(Qx,t)q_{ki}\partial_k f(Qx,t)\\
&=(w\cdot\nabla f)_Q .
\end{aligned}
\end{equation}
For vector-valued functions,
\begin{equation}
\begin{aligned}
(v_Q\cdot\nabla)w_Q
&=q_{ji}v_j(Qx,t)\partial_i
(q_{lk}w_k(Qx,t))\\
&=q_{lk}(v\cdot\nabla w)_k(Qx,t)\\
&=((v\cdot\nabla)w)_Q .
\end{aligned}
\end{equation}
Moreover,
\begin{equation}
\begin{aligned}
(v_Q)_t&=(v_t)_Q,\\
\nabla f_Q&=(\nabla f)_Q,\\
\nabla(\nabla\cdot v_Q)&=(\nabla(\nabla\cdot v))_Q,\\
\Delta v_Q&=(\Delta v)_Q .
\end{aligned}
\end{equation}
Therefore, every term in \eqref{eq tau u} is preserved under the rotation, which proves
the lemma.
\end{proof}

We denote by $\Gamma$ any operator in the collection of spatial and
rotational derivatives,
\begin{equation}
\Gamma\in
\{\partial_1,\partial_2,\partial_3,
\widetilde{\Omega}_{12},
\widetilde{\Omega}_{13},
\widetilde{\Omega}_{23}\}.
\end{equation}
The rotational invariance proved in Lemma \ref{lem 2.1} implies that the above
operators preserve the differential structure of the compressible
Navier-Stokes system. In particular, they commute with the differential
operators and Helmholtz projections $\mathbb{P}_{cf}, \mathbb{P}_{df}$ involved in the analysis, and we have
\begin{equation}
[\Gamma,\mathbb{P}_{\rm cf}]=[\Gamma,\mathbb{P}_{\rm df}]
=[\Gamma,\operatorname{div}]=0 .
\end{equation}
These commutation properties allow us to apply $\Gamma$ repeatedly to the
system while preserving the structure of the principal part. Define the operators $\Gamma^k$ by
\begin{equation}
\begin{aligned}
\Gamma^{N}
:=
\left\{
\partial^\beta\widetilde{\Omega}^{k}:
|\beta|+|k|\leq k,\quad |k|\leq4
\right\},
\end{aligned}
\end{equation}
where
\[
\partial^\beta
=\partial_1^{\beta_1}\partial_2^{\beta_2}\partial_3^{\beta_3},
\qquad
\widetilde{\Omega}^{k}
=
\widetilde{\Omega}_{12}^{\alpha_1}
\widetilde{\Omega}_{13}^{\alpha_2}
\widetilde{\Omega}_{23}^{\alpha_3}.
\]
Applying a composition of the operators $\Gamma$ to equations \eqref{eq tau u} and
using the commutation relations, we obtain 
\begin{equation}\label{eq Gamma tau u}
\begin{aligned}
\left\lbrace
\begin{array}{lr}
(\Gamma^k \tau)_t + b  \sum\limits_{|c|+|d| = k}  \left[\Gamma^c(\tau+1) \nabla \cdot \Gamma^d u + \Gamma^c u \cdot \nabla \Gamma^d \tau \right]= 0, &\\ 
(\Gamma^k u)_t + \sum\limits_{|c|+|d| = k} \left[\Gamma^c u \cdot \nabla \Gamma^d u + b^{-1} \Gamma^c (\tau+1) \nabla \Gamma^d \tau \right]&\\
 = \sum\limits_{|c|+|d| = k} \Gamma^c(\rho^{-1}) [{\mu} \Delta \Gamma^d u + ({\lambda} + {\mu}) \nabla (\nabla \cdot \Gamma^d u)].&
\end{array}	
\right.
\end{aligned}
\end{equation}
Define $\tilde{\tau}=\Gamma^k \tau$ and $\tilde{u}=\Gamma^k u$. Separating the highest-order terms from the commutators, we have
\begin{equation}\label{eq tilde tau u}
\begin{aligned}
\left\lbrace
\begin{array}{lr}
\tilde{\tau}_t + b (\tau+1) \operatorname{div} \tilde{u} + (u \cdot \nabla) \tilde{\tau} = F[\tilde{\tau}] & \\ 
\tilde{u}_t + u \cdot \nabla \tilde{u} + b^{-1} (\tau+1) \nabla \tilde{\tau} - \bar{\mu} \Delta \tilde{u} - (\bar{\lambda} + \bar{\mu}) \nabla (\operatorname{div} \tilde{u}) = F[\tilde{u}],
\end{array}	
\right.
\end{aligned}
\end{equation}
where
\begin{equation}\label{def F tau u}
\begin{aligned}
F[\tilde{\tau}] &:= - \sum\limits_{\substack{|c|+|d| = k \\ |c|\geq 1}} \bigl( b \Gamma^c \tau \nabla \cdot \Gamma^d u + \Gamma^c u \cdot \nabla \Gamma^d \tau \bigr),  \\ 
F[\tilde{u}] &:= \sum\limits_{\substack{|c|+|d| = k \\ |c| \geq 1}} \Big\{ \Gamma^c (\rho^{-1}) [{\mu} \Delta \Gamma^d u + ({\lambda} + {\mu}) \nabla (\operatorname{div} \Gamma^d u)] \\
&\quad - \Gamma^c u \cdot \nabla \Gamma^d u - b^{-1} \Gamma^c \tau \nabla \Gamma^d \tau \Big\}.
\end{aligned}
\end{equation}
In fact, the nonlinear terms $F[\widetilde{\tau}]$ and $F[\widetilde{u}]$ contain only
lower-order commutator terms.

Note that the third term on the left-hand side of \eqref{eq Gamma tau u}$_2$ satisfies
\begin{equation}
    \sum_{|c|+|d| = k}\nabla \times \left[b^{-1} \Gamma^c (\tau+1) \nabla \Gamma^d \tau\right] = 0.
\end{equation} 
Recall the vector identities
\begin{equation}
\begin{aligned}
(w \cdot \nabla) v + (v \cdot \nabla )w & = - \left[w \times (\nabla \times v) + v \times (\nabla \times w)\right] + \nabla(v \cdot w), \\
\nabla \times(w \times w^\prime) &= (w^\prime \cdot \nabla ) w - (w \cdot \nabla ) w^\prime +  (\nabla \cdot w^\prime) w- (\nabla \cdot w) w^\prime.
\end{aligned}
\end{equation}
Then the second term on the left-hand side of \eqref{eq Gamma tau u}$_2$ obeys
\begin{equation}
\begin{aligned}
 & \quad (\Gamma^c u \cdot \nabla) \Gamma^d u + (\Gamma^d u \cdot \nabla) \Gamma^c u \\
 &= - [ \Gamma^c u \times (\nabla \times \Gamma^d u) + \Gamma^d u \times (\nabla \times \Gamma^c u) ] + \nabla( \Gamma^c u \cdot \Gamma^d u ).
 \end{aligned}
 \end{equation}
Consequently,
\begin{equation}
\begin{aligned}
& \sum\limits_{|c|+|d| =k} \nabla \times \frac{1}{2} [(\Gamma^c u \cdot \nabla) \Gamma^d u + (\Gamma^d u \cdot \nabla) \Gamma^c u] \\
& = - \sum\limits_{|c|+|d| = k} \nabla \times \frac{1}{2} \{ [ \Gamma^c u \times (\nabla \times \Gamma^d u) + \Gamma^d u \times (\nabla \times \Gamma^c u) ] + \nabla( \Gamma^c u \cdot \Gamma^d u ) \} \\
& = -\sum\limits_{|c|+|d| = k} \frac{1}{2} [ \nabla \times (\Gamma^c u \times (\nabla \times \Gamma^d u)) + \nabla \times (\Gamma^d u \times (\nabla \times \Gamma^c u)) ] \\
& = -\sum\limits_{|c|+|d| = k} [\nabla \times (\Gamma^c u \times (\nabla \times \Gamma^d u))] \\
& = - (( \nabla \times \Gamma^d u ) \cdot \nabla ) \Gamma^c u + ( \Gamma^c u \cdot \nabla ) (\nabla \times \Gamma^d u ) + (\nabla \cdot \Gamma^c u ) (\nabla \times \Gamma^d u ).
\end{aligned}
\end{equation}
For the right-hand side of \eqref{eq Gamma tau u}$_2$, we have
\begin{equation}
\begin{aligned}
{\mu} \Delta \Gamma^d u + ({\lambda} + {\mu}) \nabla (\nabla \cdot \Gamma^d u) = \nu \nabla (\nabla \cdot \Gamma^d u) - \mu \nabla \times (\nabla \times \Gamma^d u).
\end{aligned}
\end{equation}
Combining the preceding computations, we obtain the following equations for $\tilde{m}:=\operatorname{curl} \tilde{u}$ and $\tilde{n}:=\operatorname{div} \tilde{u}$:
\begin{equation}\label{eq tilde n m}
\begin{aligned}
\left\lbrace
\begin{array}{lr}
\tilde{n}_t + u \cdot \nabla \tilde{n} + b^{-1} (\tau+1) \Delta \tilde{\tau} - \bar{\nu} \Delta \tilde{n} = F[\tilde{n}], \\ 
\tilde{m}_t - \bar{\mu} \Delta \tilde{m} + (u \cdot \nabla) \tilde{m} = F[\tilde{m}],
\end{array}	
\right.
\end{aligned}
\end{equation}
where
\begin{equation}\label{def F n m}
\begin{aligned}
F[\tilde{n}] &: = - \sum\limits_{1\leq i,j \leq 3} \partial_i u_j \partial_j \tilde{u}_i - b^{-1}\nabla \tau \cdot \nabla \tilde{\tau} \\
& \quad - b^{-1}(\tau+1)^{-\frac{\gamma+1}{\gamma-1}} \nabla \tau \cdot \left( \mu \Delta \tilde u + (\lambda+\mu) \nabla (\operatorname{div} \tilde u)\right ) + \operatorname{div} F[\tilde{u}] , \\ 
F[\tilde{m}]& := \sum\limits_{|c|+|d| = k} \nabla \Gamma^c (\rho^{-1}) \times \left[ \nu \nabla (\nabla \cdot \Gamma^d u) - \mu \nabla \times (\nabla \times \Gamma^d u) \right] \\
&\quad + \sum\limits_{|c|+|d| = k} [ (\nabla \times \Gamma^d u) \cdot \nabla ] \Gamma^c u  - \sum\limits_{|c|+|d| = k} (\nabla \times \Gamma^d u) (\nabla \cdot \Gamma^c u) \\
&\quad + \sum\limits_{\substack{|c|+|d| = k \\ |c| \geq 1}} \Gamma^c (\rho^{-1}) {\mu} \Delta (\nabla \times \Gamma^d u)  - \sum\limits_{\substack{|c|+|d| = k \\ |c| \geq 1}} (\Gamma^c u \cdot \nabla) (\nabla \times \Gamma^d u).
\end{aligned}
\end{equation}

With the commuted equations in hand, we now collect the analytic inequalities used to estimate their nonlinear terms.

\subsection{Some Useful Inequalities}

In this section, we collect several Sobolev-type estimates that will be
frequently used in the nonlinear energy estimates. We use the polar coordinates
\[
x=r\omega,\qquad r=|x|,\quad \omega\in\mathbb S^2 .
\]
The variables $r$ and $\omega$ represent the radial and angular directions,
respectively. 

\medskip
\noindent
\textbf{(1) Sobolev inequalities on the unit sphere.}

\begin{lem}[\cite{Kal85}]
For a sufficiently smooth function $f$, the angular
derivatives generated by the rotational vector fields yield
\begin{equation} \label{sobolev}
\|f\|_{L^\infty_\omega}
\lesssim
\|\widetilde{\Omega}^{\leq 2}f\|_{L^2_\omega},
\qquad
L^p_\omega:=L^p_{d\omega}(\mathbb S^2).
\end{equation}
\end{lem}

\medskip
\noindent
\textbf{(2) Hardy's inequality.}

\begin{lem}[\cite{Chemin11}]

For $0\leq s<\frac32$, the following inequality holds for a sufficiently smooth function $f$:
\begin{equation}\label{hardy}
\|r^{-s}f\|_{L^2_x(\mathbb R^3)}
\lesssim
\|f\|_{\dot H^s(\mathbb R^3)} .
\end{equation}
\end{lem}

\medskip
\noindent
\textbf{(3) Trace-type inequalities.}

This estimate concerns the trace of a function on the spheres
$\{r=\mathrm{const}\}$. For a smooth function $f$, integration by parts in the
radial variable gives
\begin{equation}\label{L2L2}
        \begin{aligned}
	& \quad	\| f \|_{L^2_\omega}^2 
     \lesssim  \int_{r}^\infty \int_{\mathbb{S}^2} |\partial_{\hat{\rho}} f| \, |f| \, d\hat{\rho} d \omega \\
    & \lesssim 
		\begin{cases} 
			\|\frac{\partial_r f}{r}\|_{L^2_x} \| \frac{ f}{r} \|_{L^2_x} \lesssim  \|\nabla \nabla^{\leq 1} f\|_{L^2_x ( \{\hat{\rho} \geq r\} )}^2, \\ 
			r^{-1} \|\partial_r f\|_{L^2_x} \| \frac{ f}{r} \|_{L^2_x} \lesssim r^{-1} \|\nabla f\|_{L^2_x( \{\hat{\rho} \geq r\} )}^2, \\ 
			r^{-2} \|\partial_r f\|_{L^2_x} \| f \|_{L^2_x} \lesssim r^{-2} \|\nabla^{\leq 1} f\|_{L^2_x( \{\hat{\rho} \geq r\} )}^2.
		\end{cases}
        \end{aligned}
		\end{equation}
The resulting trace estimate is recorded in the following classical form.
\begin{lem}[\cite{Kla87}]
For a sufficiently smooth function $f$, we have
		\begin{equation} \label{ineq r f}
 \begin{aligned}   
		\| \langle r \rangle f \|_{L^\infty_r  L^2_\omega} & \lesssim \| \nabla^{\leq 2} f \|_{L^2_x}, \\
        \| \langle r \rangle^{\frac{1}{2}} f \|_{L^\infty_r  L^2_\omega} & \lesssim \| \nabla \nabla^{\leq 1} f \|_{L^2_x} , \\
 \| r^{\frac{1}{2}} f \|_{L^\infty_r  L^2_\omega} & \lesssim \| \nabla  f \|_{L^2_x}  .
 \end{aligned}
		\end{equation}
\end{lem}

We next refine the trace estimate in a form adapted to the viscosity scale.
	\begin{lem} \label{lem 2.2}
For a sufficiently smooth function $f$, we have \begin{equation}\label{eq00}
\| f \|_{ L^2_\omega} (r) \lesssim \left(\ln \varepsilon^{-1}\right)^{\frac{1}{2}} \left\| \langle r \rangle^{-\frac{1}{2}}  \nabla \nabla^{\leq 1} f \right\|_{L^2_x\left( r  \le \varepsilon^{-1}\right)} + \varepsilon^{\frac{1}{2}} \left\| \nabla f \right\|_{L^2_x\left( \mathbb{R}^3\right )}.
		\end{equation}
	
	\end{lem}
	
\begin{proof}
Let $\chi(r)$ be a smooth cut-off function such that
\begin{equation}
    \chi(r)=\begin{cases}
        0, ~~~r\geq 2;\\
        [0,1], ~~~r\in(1,2);\\
        1, ~~~r\in [0,1],
    \end{cases}
\end{equation}
	We then decompose $f$ as \( f = f \chi(r) + f(1-\chi(r)) \).
The first inequality in \eqref{L2L2} gives
\begin{equation}\label{eq000}
\begin{aligned}
\| f \chi(r) \|_{ L^2_\omega} &\lesssim \left\| \nabla \nabla^{\leq 1} (f \chi(r)) \right\|_{L^2_x} \\
&\lesssim \left\|  \langle r \rangle^{-\frac{1}{2}}  \nabla \nabla^{\leq 1} f \right\|_{L^2_x( r\le 2 )} + \sup_{1\le r\le 2}\| f  \|_{ L^2_\omega}.	
\end{aligned}
\end{equation}

It remains to control \( \| f \|_{ L^2_\omega} (r)\) for $r\ge 1$. First, note that
	
	\begin{equation}\label{eq01}
	\int_1^{\varepsilon^{-1}} \int_{\mathbb{S}^2} \frac{|f|^2}{\hat{\rho}} d \hat{\rho} d\omega
\le \ln \varepsilon^{-1} \| f \|^2_{ L^2_\omega } (\varepsilon^{-1}) + \ln \varepsilon^{-1} 	\int_1^{\varepsilon^{-1}} \int_{\mathbb{S}^2} |\partial_r f| |f | d\omega dr.
	\end{equation}
	Moreover, we have
	\begin{equation}\label{eq02}
		\begin{aligned}
	\| f \|_{ L^2_\omega}^2 (r)\Big|_{r\ge 1} &=-\int_{r(\geq 1)}^{\infty}\partial_{\hat{\rho}}\|f\|_{L_\omega^2}^2d\hat{\rho}\\
& \le\sup_{r\ge \varepsilon^{-1}}\| f \|_{L^2_\omega}^2 (r) + \int_1^{\varepsilon^{-1}} \int_{\mathbb{S}^2} |\partial_r f| |f| d\omega dr,\\
		\end{aligned}
	\end{equation}	
where the last term can be estimated by applying \eqref{eq01}:
\begin{equation}\label{eq03}
		\begin{aligned}
&\int_1^{\varepsilon^{-1}} \int_{\mathbb{S}^2} |\partial_r f| |f|  d\omega dr\\
\le& 2\ln \varepsilon^{-1}\int_1^{\varepsilon^{-1}} \int_{\mathbb{S}^2}|\partial_rf|^2rd\omega dr+\frac{1}{2\ln \varepsilon^{-1}}\int_1^{\varepsilon^{-1}} \int_{\mathbb{S}^2}\frac{|f|^2}{r}d\omega dr\\
\le& 2\ln \varepsilon^{-1}\int_1^{\varepsilon^{-1}} \int_{\mathbb{S}^2}|\partial_rf|^2rd\omega dr+\frac12\|f\|_{L_\omega^2}^2(\varepsilon^{-1})+\frac12\int_1^{\varepsilon^{-1}} \int_{\mathbb{S}^2} |\partial_r f| |f| d\omega dr,\\
		\end{aligned}
	\end{equation}	
which implies that
\begin{equation}\label{eq04}
		\begin{aligned}
&\int_1^{\varepsilon^{-1}} \int_{\mathbb{S}^2} |\partial_r f| |f| dr d\omega\\
\le& 4\ln \varepsilon^{-1}\int_1^{\varepsilon^{-1}} \int_{\mathbb{S}^2}|\partial_rf|^2rd\omega dr+\|f\|_{L_\omega^2}^2(\varepsilon^{-1})\\
\le &4\ln \varepsilon^{-1}\left\|r^{-\frac12}\nabla f\right\|_{L^2(r\ge 1)}^2
+\|f\|_{L_\omega^2}^2(\varepsilon^{-1}).\\
		\end{aligned}
	\end{equation}	
It remains to estimate $\sup\limits_{r\ge \varepsilon^{-1}}\|f\|_{L^2_{\omega}}^2(r)$. Indeed,
\begin{equation}\label{eq05}
		\begin{aligned}
\sup_{r\ge \varepsilon^{-1}}\|f\|_{L^2_{\omega}}^2(r)\le &\int_{\varepsilon^{-1}} ^\infty\int_{\mathbb{S}^2}|\partial_rf||f|d\omega dr\\
\le &\varepsilon \int_{\varepsilon^{-1}} ^\infty\int_{\mathbb{S}^2}|\partial_rf|\left|\frac {f}{r}\right|r^2d\omega dr\\
\le &\varepsilon \|\nabla f\|_{L^2(\mathbb{R}^3)}^2.
\end{aligned}
	\end{equation}	
Combining \eqref{eq000}, \eqref{eq02}, \eqref{eq04} and \eqref{eq05} yields \eqref{eq00}.
\end{proof}

\medskip
\noindent
\textbf{(4) \( L_\omega^2 \)-estimate for the nonlinearities.}

We use at most four angular derivatives; accordingly,
\begin{equation} \label{Gamma N decomposition}
\Gamma^{\leq N} = \nabla \Gamma^{\leq N-1} + \tilde{\Omega}^{\leq 4} \quad \text{for } N \geq 10.
\end{equation}
We have the following estimate.

\begin{lem}\label{lem:L2omega-product}
Let $N$ be a sufficiently large integer as above. For any smooth functions
$v,w:\mathbb{R}^{3}\rightarrow\mathbb{R}$, we have
\begin{equation} \label{eq:L2omega-product}
\begin{aligned}
&
\sum_{\substack{|c|+|d|\leq N\\ |c|\geq 1}}
\left\|
\Gamma^{c}v \Gamma^{d}w
\right\|_{L_\omega^2}
\\
&\lesssim
\left\|
\Gamma^{\leq N}v
\right\|_{L_\omega^2}
\left\| 
\Gamma^{\leq [N/2]+3}w
\right\|_{L_\omega^2}
+
\left\|
\Gamma^{\leq [N/2]+3}v
\right\|_{L_\omega^2}
\left\| \Gamma^{\leq N-1}w
\right\|_{L_\omega^2}.
\end{aligned}
\end{equation}
Moreover,
\begin{equation} \label{eq:L2omega-product-2}
\begin{aligned}
&
\sum_{\substack{|c|+|d|\leq N\\ |c|\geq 1}}
\left\|
\Gamma^{c}v \Gamma^{d}w
\right\|_{L_\omega^2}
\\
&\lesssim
\left\| \nabla
\Gamma^{\leq N-1}v
\right\|_{L_\omega^2}
\left\| 
\Gamma^{\leq [N/2]+3}w
\right\|_{L_\omega^2}
+
\left\|
\Gamma^{\leq [N/2]+3}v
\right\|_{L_\omega^2}
\left\|\nabla \Gamma^{\leq N-2}w
\right\|_{L_\omega^2} \\
& \quad +\left\|
\Gamma^{\leq [N/2]+3}v
\right\|_{L_\omega^2} \left\|
\Gamma^{\leq [N/2]+3}w
\right\|_{L_\omega^2}.
\end{aligned}
\end{equation}

\end{lem}

\begin{proof}
By the definition of the commutation vector field
$\Gamma$ and \eqref{Gamma N decomposition}, every term appearing in the
left-hand side can be written as
\begin{equation}
\begin{aligned}
\nabla^{k_1}\widetilde{\Omega}^{k_2}v
\,
\nabla^{k_3}\widetilde{\Omega}^{k_4}w,
\end{aligned}
\end{equation}
where
\begin{equation}
\begin{aligned}
k_1+k_2+k_3+k_4\leq N,
\quad
k_1+k_2\geq1, 
\quad
k_2+k_4\leq4 .
\end{aligned}
\end{equation}
Throughout the proof, we repeatedly use 
\eqref{sobolev} on the unit sphere $\mathbb S^2$.

\medskip
\noindent
\textbf{Case 1: $k_2=k_4=0$.} In this case, there are no angular derivatives. If $k_1 \geq k_3$, by Hölder's inequality
on $S^2$, we have
\begin{equation}
\begin{aligned}
\|
\nabla^{k_1}v
\nabla^{k_3}w
\|_{L_\omega^2} \lesssim
\|
\nabla^{k_1}v
\|_{L_\omega^2}
\|
\nabla^{k_3}w
\|_{L_\omega^\infty}.
\end{aligned}
\end{equation}
Applying \eqref{sobolev} gives
\begin{equation}
\begin{aligned}
\|
\nabla^{k_1}v
\nabla^{k_3}w
\|_{L_\omega^2} \lesssim
\|
\nabla^{k_1}v
\|_{L_\omega^2}
\|
\nabla^{k_3} \tilde\Omega^{\leq2}w
\|_{L_\omega^2}.
\end{aligned}
\end{equation}
Since $k_3+2 \leq [N/2]+3$, this is bounded by the right-hand
side of \eqref{eq:L2omega-product}. Similarly, if $k_1 < k_3$, we have
\begin{equation}
\begin{aligned}
\|
\nabla^{k_1}v
\nabla^{k_3}w
\|_{L_\omega^2} \lesssim
\|
\nabla^{k_1}v
\|_{L_\omega^\infty}
\|
\nabla^{k_3}  w\|_{L_\omega^2} \lesssim \|
\nabla^{k_1} \tilde\Omega^{\leq2}v
\|_{L_\omega^2}
\|
\nabla^{k_3}  w
\|_{L_\omega^2} .
\end{aligned}
\end{equation}
Since $k_1+2 \leq [N/2]+3$ and $k_3 \leq N-1$, this is also bounded by the right-hand
side of \eqref{eq:L2omega-product}.

\medskip
\noindent
\textbf{Case 2: $k_2+k_4\leq2$.} Without loss of generality, assume $k_2\leq k_4$. We place the first
factor in $L_\omega^\infty$ and the second factor in $L_\omega^2$.
Hence
\begin{equation}
\begin{aligned}
& \quad
\|
\nabla^{k_1}\tilde\Omega^{k_2}v
\nabla^{k_3}\tilde\Omega^{k_4}w
\|_{L_\omega^2}
\\
&\lesssim
\|
\nabla^{k_1}\tilde\Omega^{k_2}v
\|_{L_\omega^\infty}
\|
\nabla^{k_3}\tilde\Omega^{k_4}w
\|_{L_\omega^2}
\\
&\lesssim
\|
\nabla^{k_1}\tilde\Omega^{k_2+2}v
\|_{L_\omega^2}
\|
\nabla^{k_3}\tilde\Omega^{k_4}w
\|_{L_\omega^2}.
\end{aligned}
\end{equation}
Since $k_2+2+k_4\leq4$,
the additional angular derivatives remain admissible.

\medskip
\noindent
\textbf{Case 3: $2<k_2+k_4\leq4$.} We distinguish two subcases.

\smallskip
\noindent
\emph{Case 3.1: $k_2\leq k_4$.} As in Case 2, we have
\begin{equation}
\begin{aligned}
& \quad
\|
\nabla^{k_1}\tilde\Omega^{k_2}v
\nabla^{k_3}\tilde\Omega^{k_4}w
\|_{L_\omega^2}
\\
&\lesssim
\|
\nabla^{k_1}\tilde\Omega^{k_2+2}v
\|_{L_\omega^2}
\|
\nabla^{k_3}\tilde\Omega^{k_4}w
\|_{L_\omega^2}.
\end{aligned}
\end{equation}
Since $k_2+2\le4$ and $k_1+k_2+2 \leq N-k_3-k_4 +2 \leq N$, the angular derivative order is still admissible.

\smallskip
\noindent
\emph{Case 3.2: $k_2>k_4$.} If
\begin{equation}
\begin{aligned}
k_1+k_2\geq k_3+k_4,
\end{aligned}
\end{equation}
we put the second factor in $L_\omega^\infty$:
\begin{equation}
\begin{aligned}
&
\|
\nabla^{k_1}\tilde\Omega^{k_2}v
\nabla^{k_3}\tilde\Omega^{k_4}w
\|_{L_\omega^2}
\\
\lesssim&
\|
\nabla^{k_1}\tilde\Omega^{k_2}v
\|_{L_\omega^2}
\|
\nabla^{k_3}\tilde\Omega^{k_4+2}w
\|_{L_\omega^2},
\end{aligned}
\end{equation}
which is again acceptable since $k_4+2<4$ and $k_3+k_4+2 \leq [N/2]+3$.

Otherwise, if $k_1+k_2<k_3+k_4$, we distinguish two further subcases. When $k_2\le3$, Hölder's inequality together with the Sobolev embedding yields
\[
\begin{aligned}
& \quad
\left\|\nabla^{k_1}\widetilde\Omega^{k_2}v\,
\nabla^{k_3}\widetilde\Omega^{k_4}w\right\|_{L^2_\omega}
\\
& \lesssim \left\|\nabla^{k_1}\widetilde\Omega^{k_2}v\right\|_{L^4_\omega} 
\left\|\nabla^{k_3}\widetilde\Omega^{k_4}w\right\|_{L^4_\omega} \\
&
\lesssim
\left\|\nabla^{\le k_1}\widetilde\Omega^{\le k_2+1}v\right\|_{L^2_\omega}
\,
\left\|\nabla^{\le k_3}\widetilde\Omega^{\le k_4+1}w\right\|_{L^2_\omega},
\end{aligned}
\]
which is bounded by the right-hand side of \eqref{eq:L2omega-product}. Finally, when $k_2=4$, we apply \eqref{sobolev} to the
second factor and obtain
\[
\begin{aligned}
& \quad
\left\|\nabla^{k_1}\widetilde\Omega^{4}v\,
\nabla^{k_3}\widetilde\Omega^{k_4}w\right\|_{L^2_\omega}
\\
&
\lesssim
\left\|\nabla^{k_1}\widetilde\Omega^{4}v\right\|_{L^2_\omega}
\,
\left\|\nabla^{\le k_3}\widetilde\Omega^{\le k_4+2}w\right\|_{L^2_\omega},
\end{aligned}
\]
and the derivative count is again admissible due to
$k_4+2\le4$. Collecting the preceding estimates yields
\eqref{eq:L2omega-product}. The estimate
\eqref{eq:L2omega-product-2} follows immediately from the decomposition \eqref{Gamma N decomposition}.

\end{proof}

	\section{Energy Estimates}
	\subsection{Energy Estimates for $(\tilde{\tau},\tilde{u})$}

We first establish the basic energy estimate for the coupled system \eqref{eq tilde tau u}.
The following lemma provides the main $L^2$ energy inequality, which will be
used in the subsequent high-order energy estimates.

\begin{lem}\label{lem:energy_tau_u}
 Denote $\tilde\tau=\Gamma^k \tau$ and $\tilde u=\Gamma^k  u$. Let $(\tilde\tau,\tilde u)$ be a smooth solution of the system \eqref{eq tilde tau u}. Then, for any $T>0$, we have
\begin{equation} \label{energy}
\begin{aligned}
&\quad \sup_{0\le t\le T}
\int_{\mathbb{R}^{3}}
\left(
|\tilde\tau(t,x)|^{2}
+b^{2}|\tilde u(t,x)|^{2}
\right)dx
\\
&\quad
+\int_{0}^{T}\int_{\mathbb{R}^{3}}
\left(
\nu |\nabla\tilde u_{cf}|^{2}
+\mu |\nabla\tilde u_{df}|^{2}
\right)dxdt
\\
&\lesssim
\|\tilde\tau(0)\|_{L^{2}_{x}}^{2}
+\|\tilde u(0)\|_{L^{2}_{x}}^{2}
+\mathcal E_N(T),
\end{aligned}
\end{equation}
where
\begin{equation} \label{3.3}
\begin{aligned}
\mathcal E_N(T)
:=&
\int_{0}^{T}\int_{\mathbb{R}^{3}}
\Big\{
(|\nabla\tau|+|\nabla\cdot u|)
(|\tilde\tau|^{2}+|\tilde u|^{2})
\\
&\qquad
+\mu\nabla\tilde u\cdot\nabla\tau\cdot\tilde u
+(\lambda+\mu)
(\nabla\cdot\tilde u)(\nabla\tau\cdot\tilde u)
\\
&\qquad
+\tilde u\cdot F[\tilde u]
+\tilde\tau F[\tilde\tau]
\Big\}
dxdt .
\end{aligned}
\end{equation}
\end{lem}

\begin{proof}
Taking the $L^2(\mathbb R^3)$ inner product of
\eqref{eq tilde tau u}$_1$ with $\tilde\tau$ and
\eqref{eq tilde tau u}$_2$ with $b^2\tilde u$, respectively, and adding the identities together, we obtain
\begin{equation}\label{3.1}
\begin{aligned}
&\frac12\frac{d}{dt}
\left(
|\tilde\tau|^2+b^2|\tilde u|^2
\right)
-b^2\bar\mu \tilde u \Delta \tilde u
-b^2(\bar\lambda+\bar\mu)\tilde u\nabla(\nabla\cdot\tilde u)
\\
={}&
-b(\tau+1)
\bigl(
\tilde\tau\,\nabla\cdot\tilde u
+\tilde u\cdot\nabla\tilde\tau
\bigr)
-b^2\tilde u\cdot(u\cdot\nabla)\tilde u
-\tilde\tau(u\cdot\nabla)\tilde\tau
\\
&+b^2\tilde u\cdot F[\tilde u]
+\tilde\tau\,F[\tilde\tau].
\end{aligned}
\end{equation}
Integrating \eqref{3.1} over $\mathbb R^3$, we have
\begin{equation} \label{3.2}
\begin{aligned}
&\frac12\frac{d}{dt} \int_{\mathbb{R}^3}
\left(
|\tilde\tau|^2+b^2|\tilde u|^2
\right) dx
+b^2 \int_{\mathbb R^3} L \tilde u \cdot \tilde u dx
\\
={}& \int_{\mathbb{R}^3} 
-\left[b(\tau+1)
\bigl(
\tilde\tau\,\nabla\cdot\tilde u
+\tilde u\cdot\nabla\tilde\tau
\bigr)
+b^2\tilde u\cdot(u\cdot\nabla)\tilde u
+\tilde\tau(u\cdot\nabla)\tilde\tau \right] dx \\
&+ \int_{\mathbb{R}^3}  \left(b^2\tilde u\cdot F[\tilde u]
+\tilde\tau\,F[\tilde\tau] \right) dx \\
= {}& \int_{\mathbb{R}^3} 
\left[b \nabla\tau \cdot
\tilde\tau \cdot \tilde u + (\nabla \cdot u)\left(\frac{b^2}{2} |\tilde u|^2 + |\tilde \tau|^2\right) \right] dx\\
&+ \int_{\mathbb{R}^3}  \left(b^2\tilde u\cdot F[\tilde u]
+\tilde\tau\,F[\tilde\tau] \right) dx,
\end{aligned}
\end{equation}
where $L \tilde u := \bar{\mu} \Delta \tilde u + (\bar{\lambda}+\bar{\mu})\nabla (\nabla\cdot \tilde u)$ denotes the viscous operator. Integration by parts gives
\begin{equation}
\begin{aligned}
-\int_{\mathbb R^3} L \tilde u \cdot \tilde u dx 
&= \int_{\mathbb R^3} (
\bar{\mu} |\nabla \tilde u|^2
+ (\bar{\lambda}+\bar{\mu}) |\nabla\cdot \tilde u|^2) dx \\
& \quad - b^{-1} \int_{\mathbb R^3} (\tau +1)^{-\frac{\gamma+1}{\gamma-1}} \left[ \mu \nabla\tilde u\cdot\nabla\tau\cdot\tilde u+
(\lambda+\mu)
(\nabla\cdot\tilde u)
(\nabla\tau\cdot\tilde u) \right] dx,
\end{aligned}
\end{equation}
where we use \eqref{def of tau} and \eqref{def of b}. By the Helmholtz decomposition, we write
\[
\tilde u = \tilde u_{cf} + \tilde u_{df}, \qquad
\nabla\times \tilde u_{cf} = 0, \;\; \nabla\cdot \tilde u_{df} = 0.
\]
By the orthogonality of the Helmholtz decomposition, we have
\[
\|\nabla \tilde u\|_{L^2}^2 = \|\nabla \tilde u_{cf}\|_{L^2}^2 + \|\nabla \tilde u_{df}\|_{L^2}^2, \qquad
\nabla\cdot \tilde u = \nabla\cdot \tilde u_{cf}.
\]
Moreover, since $\tilde u_{cf} = \nabla \phi$ for some scalar potential $\phi$, we have
\[
\|\nabla \tilde u_{cf}\|_{L^2} = \|\nabla\cdot \tilde u_{cf}\|_{L^2}.
\]
Hence, the viscous energy reduces to
\begin{equation}\label{eq:dissipation}
\begin{aligned}
-\int_{\mathbb R^3} L \tilde u \cdot \tilde u dx 
&= \int_{\mathbb R^3} \left(
\bar{\nu} |\nabla \tilde u_{cf}|^2
+ \bar{\mu} |\nabla \tilde u_{df}|^2\right) dx \\
& \quad - b^{-1} \int_{\mathbb R^3} (\tau +1)^{-\frac{\gamma+1}{\gamma-1}} \left[ \mu \nabla\tilde u\cdot\nabla\tau\cdot\tilde u+
(\lambda+\mu)
(\nabla\cdot\tilde u)
(\nabla\tau\cdot\tilde u) \right] dx.
\end{aligned}
\end{equation}
By Remark \ref{rem14},
$ (\tau +1)^{-\frac{\gamma+1}{\gamma-1}} \in ( 2^{-\frac{\gamma+1}{\gamma-1}}, 2^{\frac{\gamma+1}{\gamma-1}} )$ and
$ \rho^{-1} = (\tau +1)^{-\frac{2}{\gamma-1}} \in (4^{-\frac{2}{\gamma-1}},4^{\frac{2}{\gamma-1}}) $ are bounded. Consequently, integrating \eqref{3.2} over $[0, T]$ gives \eqref{energy}-\eqref{3.3}.

\end{proof}

\subsection{Parabolic estimate for $\tilde{\tau }$}

To explore the dissipative structure for the perturbed sound speed $\tilde\tau$, we derive a parabolic-type estimate
by exploiting the coupling between $\tilde\tau$ and
$\tilde n=\nabla\cdot\tilde u$.

\begin{lem}\label{lem:parabolic_tau}
 Denote $\tilde\tau=\Gamma^k \tau$ and $\tilde u=\Gamma^k  u$. Let $(\tilde\tau,\tilde u)$ be a smooth solution of the system \eqref{eq tilde tau u}. Then, for any $T>0$, we have
\begin{equation} \label{parabolic tau}
\begin{aligned}
& \quad \int_0^T \int_{\mathbb{R}^3} |\nabla \tilde\tau|^2 dxdt \\
&\lesssim \int_0^T \int_{\mathbb{R}^3} |\tilde n|^2  dxdt+ \sup_{t\in[0,T]} \|\tilde n(t)\|_{L^2_x} \|\tilde\tau(t)\|_{L^2_x} + \mathcal P_k(T),
\end{aligned}
\end{equation}
where
\begin{equation}
\begin{aligned}
\mathcal P_k(T)
:=&
\int_0^T\int_{\mathbb R^3}
\Big(
|n||\tilde{n} ||\tilde{\tau}|
+
|\nabla\tau||\nabla\tilde{\tau}||\tilde\tau| + \nu |\nabla\tau||\tilde{\tau}||\nabla \tilde n|
\\
&\qquad
+\nu|\nabla\tilde{\tau}||\nabla\tilde n|
+|\tilde n F[\tilde\tau]|
+|\tilde\tau F[\tilde n]|
\Big)
dxdt .
\end{aligned}
\end{equation}

\end{lem}

 \begin{proof}
 Recall that the equations $\eqref{eq tilde tau u}_1$ and $\eqref{eq tilde n m}_1$ for $\tilde\tau$ and
$\tilde n=\nabla\cdot\tilde u$ are
\begin{equation}\label{eq:tau_n_evolution}
\begin{aligned}
\tilde\tau_t + b(\tau+1) \tilde n + (u\cdot \nabla)\tilde\tau &= F[\tilde\tau], \\
\tilde n_t + u\cdot \nabla \tilde n + b^{-1} (\tau+1) \Delta \tilde\tau - \bar\nu \Delta \tilde n &= F[\tilde n].
\end{aligned}
\end{equation}
Multiplying the first equation by \(\tilde n\) and the second by \(\tilde\tau\), adding the identities together and integrating over \(\mathbb{R}^3\), we obtain
\begin{equation}\label{eq:parabolic_energy}
\begin{aligned}
&\quad \frac{d}{dt} \int_{\mathbb{R}^3} \tilde n \tilde\tau \,dx + \int_{\mathbb{R}^3} b(\tau+1) |\tilde n|^2 \,dx + \int_{\mathbb{R}^3} b^{-1} (\tau+1) \tilde\tau \Delta \tilde\tau \,dx - \int_{\mathbb{R}^3} \bar\nu \tilde\tau \Delta \tilde n \,dx \\
& = \int_{\mathbb{R}^3} \big[ -(u\cdot \nabla)(\tilde n \tilde\tau) + \tilde n F[\tilde\tau] + \tilde\tau F[\tilde n] \big] \,dx.
\end{aligned}
\end{equation}
By integration by parts and \eqref{def of tau}, we get
\begin{equation}
\begin{aligned}
 \int_{\mathbb{R}^3} (\tau+1) \tilde\tau \Delta \tilde\tau dx & = -\int_{\mathbb{R}^3} \nabla \tau \cdot \nabla \tilde{\tau}\cdot \tilde{\tau} dx  - \int_{\mathbb{R}^3} (\tau+1) |\nabla \tilde{\tau}|^2 dx,
\end{aligned}
\end{equation}
and
\begin{equation}
\begin{aligned}
- \int_{\mathbb{R}^3} \bar\nu \tilde\tau \Delta \tilde n dx & = \int_{\mathbb{R}^3} \nu \rho^{-1} \left[ - b^{-1} (\tau+1)^{-1} \nabla \tau \cdot \tilde\tau \cdot \nabla \tilde n + \nabla \tilde\tau \cdot \nabla \tilde n\right ] dx.
\end{aligned}
\end{equation}
By Remark \ref{rem14}, $\tau +1 \in (\frac{1}{2},2) $ and $ \rho^{-1} = (\tau +1)^{-\frac{2}{\gamma-1}} \in (4^{-\frac{2}{\gamma-1}},4^{\frac{2}{\gamma-1}}) $. Integrating \eqref{eq:parabolic_energy} over
$[0,T]$ therefore yields \eqref{parabolic tau}.

 \end{proof}

\subsection{Morawetz estimate}

The previous energy estimates provide the control of the standard $L^2$
energy. To obtain additional local decay and weighted spacetime estimates,
we establish a Morawetz-type estimate.

\begin{lem}\label{lem:Morawetz}
Let $(\tilde\tau,\tilde u)$ be a smooth solution of the system
\eqref{eq tilde tau u}. Then, for any $T>0$, the following
Morawetz-type estimate holds:
\begin{equation}\label{eq:Morawetz-final}
\begin{aligned}
& \quad
\int_0^T \int_{ \mathbb R^3}
\left(
\frac{|\nabla\tilde\tau|^2+|\tilde n|^2}
{\langle r \rangle}
+
\frac{|\tilde\tau|^2}
{\langle r \rangle^2 r}
\right)
dxdt
\\
&
\lesssim |\ln\varepsilon^{-1}|\left \| \Gamma^{\leq 5} ( \tau , u )\right\|_{L^\infty(0,T;L^2_x)} \int_0^T \int_{ \mathbb R^3} \left(
\frac{|\nabla\tilde\tau|^2+|\tilde n|^2}
{\langle r \rangle}
+
\frac{|\tilde\tau|^2}
{\langle r \rangle^2 r}
\right) dxdt \\
& \quad
+
|\ln\varepsilon^{-1}|
\,\nu \int_0^T \int_{ \mathbb R^3}
\bigl(
|\nabla\tilde\tau|^2
+
|\Delta\tilde n|^2
\bigr)\,dxdt \\
& \quad + |\ln\varepsilon^{-1}| \sup\limits_{t\in [0,T]}
\|\nabla\tilde\tau\|_{  L_x^2}
\|\nabla\tilde u_{cf}\|_{ L_x^2} 
+ \varepsilon \int_0^T \int_{ \mathbb R^3}
\left(
|\nabla\tilde\tau|^2
+
|\nabla\tilde u_{cf}|^2
\right)dxdt \\
& \quad
+
|\ln\varepsilon^{-1}| \int_0^T \int_{ \mathbb R^3}  \left|\tilde n\,(\partial_rF[\tilde\tau]+r^{-1}F[\tilde\tau]) 
+ (\tilde\tau_r+r^{-1}\tilde\tau)
F[\tilde n] \right|  dxdt.
\end{aligned}
\end{equation}

\end{lem}

\begin{proof}
Recall that $(\tilde\tau,\tilde n)$ satisfies \eqref{eq tilde tau u} and \eqref{eq tilde n m}, respectively. Differentiating the equation for $\tilde\tau$ in the radial
direction yields
\begin{equation}\label{eq:tau-radial}
(r\tilde\tau)_{tr}
+b r \tau_r \tilde n + b(\tau+1) (r\tilde n)_r + (r\,u\cdot\nabla\tilde\tau)_r
=r\partial_rF[\tilde\tau]+F[\tilde\tau].
\end{equation}
The basic observation is that, after coupling
\eqref{eq:tau-radial} with the equation for $\tilde n$,
one obtains a positive contribution controlling both
$\tilde\tau$ and $\tilde n$. To this end, we multiply \eqref{eq:tau-radial} and \eqref{eq tilde n m} by $r\tilde n$
and
$r\partial_r(r\tilde\tau)$, respectively.
Adding the identities together yields
\begin{equation}\label{eq:Morawetz-coupled}
\begin{aligned}
&\quad (r \tilde n) \Big[  (r\tilde\tau)_{tr} + b r \tau_r \tilde n + b(\tau+1) (r\tilde n)_r + (ru\cdot\nabla\tilde\tau)_r - r\partial_r F[\tilde\tau] - F[\tilde\tau] \Big] \\
& + r(r\tilde\tau)_r \Big[ \tilde n_t + u\cdot\nabla \tilde n + b^{-1} (\tau+1)\Delta\tilde\tau - \bar\nu \Delta \tilde n - F[\tilde n] \Big] = 0.
\end{aligned}
\end{equation}
After integration by parts in $r$, it follows that
\begin{equation}\label{eq:Morawetz-integrated}
\begin{aligned}
&  \frac{d}{dt} \int_{\mathbb{S}^2} r \tilde n \, (r\tilde\tau)_r \, d\omega
+ b \int_{\mathbb{S}^2} \partial_r \tau \, |r\tilde n|^2 \, d\omega
+ \frac{ b}{2} \int_{\mathbb{S}^2} (\tau+1) \partial_r |r \tilde n|^2 \, d\omega \\
& + \frac{b^{-1}}{2}  \int_{\mathbb{S}^2} (\tau+1) \partial_r |(r\tilde\tau)_r|^2 \, d\omega
- \frac{b^{-1}}{2}  \int_{\mathbb{S}^2} (\tau+1) r^{-2} \partial_r |\nabla_{S^2} (r\tilde\tau)|^2 \, d\omega \\
&-b^{-1} \int_{\mathbb{S}^2} \nabla_{S^2} \tau \cdot r^{-2} (r\tilde\tau)_r \nabla_{S^2} (r\tilde\tau)\, d\omega-\int_{\mathbb{S}^2} r(r\tilde\tau)_r \bar\nu \Delta \tilde n d\omega \\
& + \int_{\mathbb{S}^2} r^2 [ u \cdot \nabla (\tilde n \cdot \tilde\tau_r) + \tilde n u_r \cdot  \nabla \tilde \tau+  r^{-1} u \cdot \nabla (\tilde\tau \cdot \tilde n) ] d\omega \\
 = &\int_{\mathbb{S}^2} \big[  (r \tilde n ) (r\partial_r F[\tilde\tau] + F[\tilde\tau]) + r(r\tilde\tau)_r F[\tilde n] \big] d\omega.
\end{aligned}
\end{equation}

To localize the estimate, we introduce the standard Morawetz weight
\[
f(r)=\frac{r}{r+R},
\]
where $R\ge1$ is a dyadic parameter.
Multiplying \eqref{eq:Morawetz-integrated} by $f(r)$ and integrating in
$(t,r,\theta) \in  [0,T]\times\mathbb R \times \mathbb S^2$, we get
\begin{equation} \label{eq morawetz}
	\begin{aligned}
&   \int_{\mathbb{R}^3} f \tilde{n} \left( \tilde{\tau}_r + \frac{\tilde{\tau}}{r} \right) dx \Bigg|_{0}^{T} + \frac{b}{2}\int_{\mathbb{R}_+ \times \mathbb{R}^3} f \tau_{r}  \tilde{n}^2 dxdt - \frac{b^{-1}}{2} \int_{\mathbb{R}_+ \times \mathbb{R}^3} f \tau_{r} \left( \tilde{\tau}_r  +  \frac{\tilde{\tau}}{r} \right )^2dxdt \\
& + \frac{b^{-1}}{2} \int_0^T \int_{ \mathbb{R}^3}  (\tau+1)   \left [-f^\prime\left (|\tilde{\tau}_r|^2 + b^2 |\tilde{n}|^2\right) + (r^{-2}f)_r | \nabla_{S^2} \tilde{\tau} |^2\right]dxdt  \\
& + \frac{b^{-1}}{2} \int_0^T \int_{ \mathbb{R}^3}\frac{f^{\prime \prime}(\tau+1) + f^\prime  \tau_r - 2f^\prime r^{-1} (\tau+1) }{r} |\tilde{\tau}|^2    dxdt \\
& + \frac{b^{-1}}{2} \int_0^T \int_{ \mathbb{R}^3} r^{-2} f  \tau_{r} | \nabla_{S^2} \tilde{\tau} |^2 dxdt\\
& + b^{-1} \int_0^T \int_{ \mathbb{R}^3} (r^{-1} f)_r \nabla_{S^2}\tau \cdot  \nabla_{S^2} \tilde{\tau} \cdot  \frac{ \tilde{\tau}}{r} dxdt\\
& +\int_0^T \int_{ \mathbb{R}^3} \left[- f^{\prime}  \tilde{\tau}_r u \cdot \tilde{n}   + f\left(-\tilde{\tau}_r n \cdot \tilde{n}  + \tilde{n}  u_r \cdot \nabla \tilde{n} \right) - r^{-1} f \tilde{\tau} n \cdot \tilde{n} \right] dxdt \\
& - \int_0^T \int_{ \mathbb{R}^3} \nabla (r^{-1} f )\tilde{\tau} u  \cdot \tilde{n} dxdt - \int_{\mathbb{R}_+ \times \mathbb{R}^3} f\left ( \tilde{\tau}_r + \frac{\tilde{\tau}}{r} \right )  \bar{\nu} \Delta \tilde{n} dtdx \\
 =& \int_0^T \int_{ \mathbb{R}^3} f \left[ \tilde n \partial_r F[\tilde\tau] + \frac{\tilde n}{r}F[\tilde\tau] +\left ( \tilde{\tau}_r + \frac{\tilde{\tau}}{r}  \right) F[\tilde n] \right] dxdt,
	\end{aligned}
\end{equation}	
which can be rewritten as
\begin{equation}
	\begin{aligned}
& \quad \frac{b^{-1}}{2} \int_0^T \int_{ \mathbb{R}^3}  (\tau+1)   \left [ f^\prime\left (|\tilde{\tau}_r|^2 + b^2 |\tilde{n}|^2\right) - (r^{-2}f)_r | \nabla_{S^2} \tilde{\tau} |^2\right]dxdt  \\
& \quad + \frac{b^{-1}}{2} \int_0^T \int_{ \mathbb{R}^3}\frac{-f^{\prime \prime} +2f^\prime r^{-1}  }{r} (\tau+1)|\tilde{\tau}|^2    dxdt \\
& =  \int_{\mathbb{R}^3} f \tilde{n} \left( \tilde{\tau}_r + \frac{\tilde{\tau}}{r} \right) dx \Bigg|_{0}^{T} + \frac{b}{2}\int_{\mathbb{R}_+ \times \mathbb{R}^3} f \tau_{r}  \tilde{n}^2 dxdt - \frac{b^{-1}}{2} \int_{\mathbb{R}_+ \times \mathbb{R}^3} f \tau_{r} \left( \tilde{\tau}_r  +  \frac{\tilde{\tau}}{r} \right )^2dxdt \\
& + \frac{b^{-1}}{2} \int_0^T \int_{ \mathbb{R}^3}\frac{f^\prime  \tau_r }{r} |\tilde{\tau}|^2  dxdt  + \frac{b^{-1}}{2} \int_0^T \int_{ \mathbb{R}^3} r^{-2} f  \tau_{r} | \nabla_{S^2} \tilde{\tau} |^2 dxdt\\
& + b^{-1} \int_0^T \int_{ \mathbb{R}^3} (r^{-1} f)_r \nabla_{S^2}\tau \cdot  \nabla_{S^2} \tilde{\tau} \cdot  \frac{ \tilde{\tau}}{r} dxdt\\
& +\int_0^T \int_{ \mathbb{R}^3} \left[- f^{\prime}  \tilde{\tau}_r u \cdot \tilde{n}   + f\left(-\tilde{\tau}_r n \cdot \tilde{n}  + \tilde{n}  u_r \cdot \nabla \tilde{n} \right) - r^{-1} f \tilde{\tau} n \cdot \tilde{n} \right] dxdt \\
& - \int_0^T \int_{ \mathbb{R}^3} \nabla (r^{-1} f )\tilde{\tau} u  \cdot \tilde{n} dxdt - \int_{\mathbb{R}_+ \times \mathbb{R}^3} f\left ( \tilde{\tau}_r + \frac{\tilde{\tau}}{r} \right )  \bar{\nu} \Delta \tilde{n} dtdx \\
& + \int_0^T \int_{ \mathbb{R}^3} f \left[ \tilde n \partial_r F[\tilde\tau] + \frac{\tilde n}{r}F[\tilde\tau] +\left ( \tilde{\tau}_r + \frac{\tilde{\tau}}{r}  \right) F[\tilde n] \right] dxdt.
	\end{aligned}
\end{equation}	
By Remark \ref{rem14} and the positivity of
\[
f'(r)=\frac{R}{(r+R)^2},
 \ 
-\Big(\frac{f(r)}{r^2}\Big)_r
=\frac{1}{r^2(r+R)}
+\frac{1}{r(r+R)^2}, \ - f''(r) = \frac{2R}{(r+R)^3},
\]
the main positive terms become
\[
\int_0^T \int_{ \mathbb{R}^3} \left[ f'(|\tilde\tau_r|^2 + |\tilde n|^2) - (r^{-2} f)_r |\nabla_{S^2} \tilde\tau|^2 - \left( \frac{f''}{r} - \frac{2f'}{r^2} \right) |\tilde\tau|^2 \right] dxdt,
\]
which controls the weighted $L^2$ norms of $\tilde n$, $\nabla\tilde\tau$ and $\tilde\tau$. Therefore, \eqref{eq morawetz} implies
\begin{equation}\label{eq:localized-Morawetz}
\int_0^T \int_{\{ r \sim R \}}
\left(
\frac{|\tilde n|^2}{R}
+\frac{|\nabla\tilde\tau|^2}{R} +\frac{|\tilde\tau|^2}{R^2r}
\right)
 dxdt
\lesssim
\mathcal N_R^{left} (T)+\mathcal N_R^{right}(T),
\end{equation}
where $\mathcal N_R^{right}(T)$ denotes the contribution from the right-hand side of \eqref{eq morawetz}, i.e.,
\begin{equation} \label{def of N right}
\mathcal N_R^{right}(T) = \int_0^T \int_{ \{ r \sim R \} } \left| \tilde n (\partial_rF[\tilde\tau]+r^{-1}F[\tilde\tau]) +  (\tilde\tau_r+r^{-1}\tilde\tau)
F[\tilde n]  \right| dxdt,
\end{equation}
and $\mathcal N_R^{left}(T)$ consists of the remaining terms on the left-hand side of \eqref{eq morawetz}, i.e.,
\begin{equation} \label{error terms}
	\begin{aligned}
\mathcal N_R^{left}(T) = \sum\limits_{i=1}^8 \mathcal N_R^{i}(T),
	\end{aligned}
\end{equation}	
where
\begin{equation} \label{error terms}
	\begin{aligned}
\mathcal N_R^{1}(T) = & \int_0^T \int_{ \{ r \sim R \} } f \tau_r ( |\tilde{n}|^2 + |\nabla \tilde{\tau}|^2 ) dxdt, \\
\mathcal N_R^{2}(T) = & \int_0^T \int_{ \{ r \sim R \} } f \tau_r  \left| \frac{\tilde{\tau}}{r} \right|^2 dxdt, \\
\mathcal N_R^{3}(T) = & \int_0^T \int_{ \{ r \sim R \} } \frac{1}{r(r+R)} \left( \nabla_{S^2}\tau \nabla_{S^2} \tilde{\tau}  \frac{\tilde{\tau}}{r}+  \tau_r |\nabla_{S^2} \tilde{\tau}|^2  \right) dxdt, \\ 
\mathcal N_R^{4}(T) =& \int_0^T \int_{ \{ r \sim R \} } f (\tilde{\tau}_r n \cdot \tilde{n} + \tilde{n} u_r  \cdot \nabla \tilde{\tau}) dxdt, \\
\mathcal N_R^{5}(T) =& \int_0^T \int_{ \{ r \sim R \} }  \frac{1}{r+R} (  \tilde{\tau}_r u \cdot \tilde{n}+\tilde{\tau} n \cdot \tilde{n} ) dxdt , \\ \mathcal N_R^{6}(T) =& \int_0^T \int_{ \{ r \sim R \} } \frac{1}{(r+R)^2} \tilde{\tau} u  \cdot \tilde{n} dxdt, \\
\mathcal N_R^{7}(T) =& \int_0^T \int_{ \{ r \sim R \} } f \left(\tilde{\tau}_r + \frac{\tilde{\tau}}{r} \right) \cdot \bar{\nu} \Delta \tilde{n} dxdt, \\
\mathcal N_R^{8}(T) =& \left. \int_{ \{ r \sim R \} } f \tilde{n} \left(\tilde{\tau}_r + \frac{\tilde{\tau}}{r} \right)  dx \right|_0^T .
	\end{aligned}
\end{equation}	
These terms are treated as errors and estimated using \eqref{sobolev}, \eqref{hardy} and the bootstrap assumptions. We estimate each term in $\mathcal N_R^{left}(T)$ in turn.

\medskip
\noindent
{\bf Estimate of $\mathcal N_R^{1}(T)$.}  By \eqref{ineq r f} and \eqref{sobolev}, we have
\begin{equation}
\begin{aligned}
\| \langle r \rangle  \tau_r \|_{ L^\infty_t L^\infty_x } & \lesssim \| \langle r \rangle \Gamma^{\le 2} \tau_r \|_{ L^\infty_tL^\infty_r L^2_\omega } \\   & \lesssim \|  \Gamma^{\le 5} \tau  \|_{L^\infty_t   L^2_x},
\end{aligned}
\end{equation}
and thus
\begin{equation} \label{N left 1}
\begin{aligned}
 \mathcal N_R^{1}(T) \lesssim \| \Gamma^{\leq 5}  \tau\|_{L^\infty(0,T;L^2_x)} \int_0^T \int_{ \{ r \sim R \}} \langle r \rangle^{-1} (|\tilde{n}|^2 + |\nabla \tilde{\tau}|^2 ) dxdt .    
\end{aligned}
\end{equation}

\medskip
\noindent
{\bf Estimate of $\mathcal N_R^{2}(T)$.} For $r \geq 1$, we have
\begin{equation}
\begin{aligned}
 &\quad \int_0^T \int_{ \{ r \sim R \} }   |\tau_r  |\left| \frac{\tilde{\tau}}{r} \right|^2 dxdt \\ 
 &\lesssim \| \Gamma^{\leq 5}  \tau\|_{L^\infty(0,T;L^2_x)} \int_0^T \int_{ \{ r \sim R \}} \langle r \rangle^{-2} r^{-1} |\tilde{\tau}|^2  dxdt .    
\end{aligned}
\end{equation}
For $r < 1$, using \eqref{hardy}, we have
\begin{equation}
\begin{aligned}
 &\quad \int_0^T \int_{ \{ r <1 \} }   |\tau_r  |\left| \frac{\tilde{\tau}}{r} \right|^2 dxdt \\
 & \lesssim \| \Gamma^{\leq 5}  \tau\|_{L^\infty(0,T;L^2_x)} \int_0^T \int_{   \{ r <1 \}}  |\nabla \tilde{\tau}|^2   dxdt .    
\end{aligned}
\end{equation}

\medskip
\noindent
{\bf Estimate of $\mathcal N_R^{3}(T)$.}  Noting $|\nabla_{S^2}\tau| \lesssim r |\nabla \tau|$, we have
\begin{equation}
\begin{aligned}
\mathcal N_R^{3}(T) \lesssim \int_0^T \int_{ \{ r \sim R \}} \frac1{r+R}
|\nabla \tau|
|\nabla_{S^2}\tilde\tau|
\left(
|\nabla\tilde\tau|+\frac{|\tilde\tau|}{r}
\right) dxdt \\
\end{aligned}
\end{equation}
For $r \geq 1$,
we have $(r+R)^{-1}\sim r^{-1} \  (r \sim R)$. Applying \eqref{ineq r f} and \eqref{sobolev} yields
\begin{equation}
\begin{aligned}
\| \langle r \rangle  \nabla \tau \|_{   L^\infty_x }  \lesssim \|  \Gamma^{\le 5} \tau  \|_{ L^2_x} .
\end{aligned}
\end{equation}
Hence,
\begin{equation}
\begin{aligned}
\mathcal N_R^{3}(T) &\lesssim \| \Gamma^{\leq 5}  \tau\|_{L^\infty(0,T;L^2_x)} \int_0^T \int_{ \{ r \sim R \}}  \langle r \rangle ^{-1}
\left(
|\nabla\tilde\tau|^2+\frac{|\tilde\tau|^2}{r^2}
\right) dxdt .    
\end{aligned}
\end{equation}
For $r < 1$, by \eqref{hardy},
we have
\begin{equation}
\begin{aligned}
& \quad \int_0^T \int_{ \{ r <1 \}} \frac1{r+1}
|\nabla \tau|
|\nabla_{S^2}\tilde\tau|
\left(
|\nabla\tilde\tau|+\frac{|\tilde\tau|}{r}
\right) dxdt \\
&\lesssim  \int_0^T \int_{ \{ r <1 \}}  |\nabla \tau|
|\nabla \tilde\tau|
\left(
|\nabla\tilde\tau|+\frac{|\tilde\tau|}{r}
\right) dxdt \\
&\lesssim \| \Gamma^{\leq 3}  \tau\|_{L^\infty(0,T;L^2_x)} \int_0^T \int_{ \{ r <1 \}} 
|\nabla \tilde\tau|
\left(
|\nabla\tilde\tau|+\frac{|\tilde\tau|}{r}
\right) dxdt \\
&\lesssim \| \Gamma^{\leq 3}  \tau\|_{L^\infty(0,T;L^2_x)} \int_0^T \int_{ \{ r <1 \}} |\nabla\tilde\tau|^2  dxdt .
\end{aligned}
\end{equation}

\medskip
\noindent
{\bf Estimate of $\mathcal N_R^{4}(T)$.} Noting $ n = \nabla \cdot u $, we also have
\begin{equation}
\begin{aligned}
\mathcal N_R^{4}(T)
 &\lesssim \|\Gamma^{\leq 5}  u\|_{L^\infty(0,T;L^2_x)} \int_0^T \int_{ \{ r \sim R \}} \langle r \rangle^{-1} |\tilde{n}| |\nabla \tilde{\tau}|  dxdt   \\
  &\lesssim \|\Gamma^{\leq 5}  u\|_{L^\infty(0,T;L^2_x)} \int_0^T \int_{ \{ r \sim R \}}  \langle r \rangle^{-1} \left(|\tilde{n}|^2 + |\nabla \tilde{\tau}|^2 \right) dxdt.
\end{aligned}
\end{equation}

\medskip
\noindent
{\bf Estimate of $\mathcal N_R^{5}(T)$.} For $r \geq 1$, we have
\begin{equation}
\begin{aligned}
& \quad \int_0^T \int_{ \{ r \sim R \}} \frac{1}{r+R} |\tilde{\tau}_r u \cdot \tilde{n}+\tilde{\tau} n \cdot \tilde{n}| dxdt \\
&\lesssim \| \Gamma^{\leq 4} u\|_{L^\infty(0,T;L^2_x)} \int_0^T \int_{ \{ r \sim R \}} \langle r \rangle^{-2} |\tilde{n}| |\nabla \tilde{\tau}| dxdt \\
& \quad + \| \Gamma^{\leq 5} u\|_{L^\infty(0,T;L^2_x)} \int_0^T \int_{ \{ r \sim R \}} \langle r \rangle^{-2} |\tilde{n}| | \tilde{\tau}| dxdt .
\end{aligned}
\end{equation}
For $r <1$, we have
\begin{equation}
\begin{aligned}
& \quad \int_0^T \int_{ \{ r <1 \}} \frac{1}{r+1} |\tilde{\tau}_r u \cdot \tilde{n}+\tilde{\tau} n \cdot \tilde{n}| dxdt \\
&\lesssim \| \Gamma^{\leq 4} u\|_{L^\infty(0,T;L^2_x)} \int_0^T \int_{ \{ r <1 \} } |\tilde{n}| |\nabla \tilde{\tau}| dxdt \\
& \quad + \| \Gamma^{\leq 5} u\|_{L^\infty(0,T;L^2_x)} \int_0^T \int_{ \{ r <1 \} }  |\tilde{n}| |\tilde{\tau}| dxdt  \\
&\lesssim \| \Gamma^{\leq 5} u\|_{L^\infty(0,T;L^2_x)} \int_0^T \int_{ \{ r <1 \} }  \left( |\tilde{n}|^2 + |\nabla \tilde{\tau}|^2 + |\tilde{\tau}|^2 \right)dxdt .
\end{aligned}
\end{equation}

\medskip
\noindent
{\bf Estimate of $\mathcal N_R^{6}(T)$.} For $r \geq 1$, we have
\begin{equation}
\begin{aligned}
& \quad \int_0^T \int_{ \{ r \sim R \}} \frac{1}{(r+R)^2} |\tilde{\tau} u  \cdot \tilde{n}| dxdt \\ &\lesssim \| \Gamma^{\leq 4} u\|_{L^\infty(0,T;L^2_x)} \int_0^T \int_{ \{ r \sim R \}} \langle r \rangle^{-3} |\tilde{n}||\tilde{\tau}| dxdt \\ &\lesssim \| \Gamma^{\leq 4} u\|_{L^\infty(0,T;L^2_x)} \int_0^T \int_{ \{ r \sim R \}} \langle r \rangle^{-3} (|\tilde{n}|^2 + |\tilde{\tau}|^2 )dxdt .
\end{aligned}
\end{equation}
For $r < 1$, we have
\begin{equation}
\begin{aligned}
& \quad \int_0^T \int_{ \{ r <1 \}} \frac{1}{(r+R)^2} |\tilde{\tau} u  \cdot \tilde{n}| dxdt \\ 
&\lesssim \| \Gamma^{\leq 4} u\|_{L^\infty(0,T;L^2_x)} \int_0^T \int_{ \{ r<1 \}} (|\tilde{n}|^2 + |\tilde{\tau}|^2 )dxdt .
\end{aligned}
\end{equation}

\medskip
\noindent
{\bf Estimate of $\mathcal N_R^{7}(T)$.} The viscous contribution in the Morawetz estimate is given by
\begin{equation}
\begin{aligned}
\mathcal N_R^{7}(T)
& \lesssim \nu \int_0^T \int_{ \{ r \sim R \}} ( |\nabla\tilde{\tau}|^2 + |\Delta\tilde{n}|^2  )    dxdt .
\end{aligned}
\end{equation}
Here, we use $\bar{\nu} = \rho^{-1} \nu $, Remark \ref{rem14} and \eqref{hardy}.

\medskip
\noindent
{\bf Estimate of $\mathcal N_R^{8}(T)$.} For the remaining boundary term, by the Helmholtz estimate $\|\tilde n\|_{L^2_x}
\lesssim
\|\nabla\tilde u_{cf}\|_{L^2_x}$ and \eqref{hardy}, we have
\[
\begin{aligned}
& \quad \sup\limits_{t\in[0,T]} \left|
\int_{ \{ r \sim R \}}
\tilde n (t)
\left(
\tilde\tau_r(t)+\frac{\tilde\tau(t)}{r}
\right)dx
\right| \\
&\lesssim \sup\limits_{t\in[0,T]}
\|\tilde n\|_{ L^2_x}
\left(
\|\nabla\tilde\tau\|_{  L^2_x}
+
\left\|\frac{\tilde\tau}{r}\right\|_{ L^2_x}
\right)
\\
&\lesssim \sup\limits_{t\in[0,T]}
\|\nabla \tilde u_{cf}\|_{ L^2_x}
\|\nabla\tilde\tau\|_{   L^2_x}.
\end{aligned}
\]

 Summing \eqref{eq:localized-Morawetz} over dyadic numbers
$1\le R\le\varepsilon^{-1}$ , we arrive at
\begin{equation}\label{eq:Morawetz-global}
\begin{aligned}
& \quad  \int_0^T \int_{ r\le\varepsilon^{-1} }
\left(
\frac{|\tilde n|^2}{\langle r \rangle}
+\frac{|\nabla\tilde\tau|^2}{\langle r \rangle}
+\frac{|\tilde\tau|^2}{\langle r \rangle^2 r}
\right)
dxdt \\
& \lesssim
|\ln\varepsilon^{-1}|
\sup_{1\le R\le\varepsilon^{-1}}
\left(\mathcal N_R^{left}(T)+\mathcal N_R^{right}(T)\right).  
\end{aligned}
\end{equation}
Moreover, for the exterior region $r\ge\varepsilon^{-1}$, we have
\[
\begin{aligned}
& \quad \int_{\mathbb R_+\times\{r\ge\varepsilon^{-1}\}}
\left(
\frac{|\nabla\tilde\tau|^2+|\tilde n|^2}{\langle r \rangle}
+\frac{|\tilde\tau|^2}{\langle r \rangle^2 r}
\right)dxdt
\\
&\lesssim
\varepsilon
\int_{\mathbb R_+\times\{r\ge\varepsilon^{-1}\}}
\left(
|\nabla\tilde\tau|^2
+|\tilde n|^2
+\frac{|\tilde\tau|^2}{r^2}
\right)dxdt .
\end{aligned}
\]
Using \eqref{hardy} and the Helmholtz estimate $\|\tilde n\|_{L^2_x}
\lesssim
\|\nabla\tilde u_{cf}\|_{L^2_x}$, we further obtain
\begin{equation} \label{morawetz r geq}
\begin{aligned}
& \quad \int_{\mathbb R_+\times\{r\ge\varepsilon^{-1}\}}
\left(
\frac{|\nabla\tilde\tau|^2+|\tilde n|^2}{\langle r \rangle}
+\frac{|\tilde\tau|^2}{\langle r \rangle^2 r}
\right)dxdt
\\
&\lesssim
\varepsilon
\int_{\mathbb R_+\times\mathbb R^3}
\left(
|\nabla\tilde\tau|^2
+
|\nabla\tilde u_{cf}|^2
\right)dxdt .
\end{aligned}
\end{equation}
Finally, combining \eqref{N left 1}-\eqref{morawetz r geq} gives \eqref{eq:Morawetz-final}.

\end{proof}

\section{Proof of main theorem}

In this section, we complete the proof of the main theorem. The main difficulty lies in controlling the nonlinear interactions generated
by the commuted system. To overcome this difficulty, we exploit the
decomposition into low and high commutation orders.

By the order splitting \eqref{def of low hi} and Lemma \ref{lem:L2omega-product}, every nonlinear term contains at least one factor of
low commutation order. The low-order component is controlled by the improved
decay and weighted spacetime estimates obtained from the Morawetz inequality,
while the high-order component is handled by the energy and parabolic
dissipation estimates.

We divide the proof into three steps. First, we derive the high-order
energy estimate for the variables $(\tau,u)$. Next, we establish the
parabolic estimate for the low-order divergence-free part of the velocity. Finally, combining the Morawetz estimate
with the previous bounds, we improve the bootstrap assumptions and close the
bootstrap argument.

\subsection{High-order energy estimate}

For $0<T<\infty$, define the energy 
\begin{equation}\label{def of E}
 E_N(T):=
 \|\Gamma^{\leq N}(\tau,u)\|^2_{L^\infty(0,T;L^2_x)}
\end{equation}
and the viscous dissipation 
\begin{equation}\label{def of D}
\begin{aligned}
D_N(T):={}&
 \nu \|\nabla\Gamma^{\leq N}u_{cf}\|^2_{L^2(0,T; L^2_x)}
+ \mu \|\nabla\Gamma^{\leq N}u_{df}\|^2_{L^2(0,T; L^2_x)} .
\end{aligned}
\end{equation}
Since $N\geq10$, we define the density dissipation by
\begin{equation}\label{def of G}
G_N(T):= \nu \|\nabla\Gamma^{\leq N-1}\tau\|^2_{L^2(0,T; L^2_x)}.
\end{equation}
Then the bootstrap assumption (BA1) gives
\begin{equation} 
\begin{aligned}
E_N^{1/2}(t)+D_N^{1/2}(t)+ M^{-1}G_N^{1/2}(t) \leq M \delta \varepsilon^{1/2}( \ln \varepsilon^{-1} )^{-1}, \quad \forall t>0.
\end{aligned}
\end{equation}
Moreover, we will use the lower-order norm
\begin{equation}\label{def of A low}
 A_j(T):=
 \varepsilon^{-1}
 \|\Gamma^{\leq j}(\tau,u)\|^2_{L^2(0,T;L^\infty_rL^2_\omega)} 
\end{equation}
as a coefficient.
Setting $N_* = [N/2]+3$, we have

\begin{lem}[High-order energy estimate]\label{lem: high-energy}
Let $N\geq 10$, and suppose that
$(\tau,u)$ is a smooth solution of \eqref{eq tau u} on $[0,T]$. Then
\begin{equation}\label{tame high energy}
 E_N(T)+  D_N(T) 
 \leq E_N(0) 
 +C  A^{1/2}_{N_{*}}(T)
 \big(  E_N(T)+  D_N(T)+  G_N(T)\big) ,
\end{equation}
where $C>0$ is independent of $M, T,\mu$, and $\nu$. Moreover, under (BA1)-(BA3), we have
\begin{equation} \label{EN DN}
\begin{aligned}
E_N^{\frac12}(T)+D_N^{\frac12}(T) \le \left(C^{\frac12} + 4 C^{\frac12} M^{\frac72} \delta^{\frac12}  \right) \delta \varepsilon^{\frac12} (\ln\varepsilon^{-1})^{-1}.
\end{aligned}
\end{equation}

\end{lem}

\begin{proof}
Summing \eqref{energy} over all commutation
multi-indices $|k|\leq N$, we obtain the high-order energy estimate
\begin{equation}
\label{eq:4.1}
\begin{aligned}
& \quad
\sup \limits_{t \in [0,T]}\|\Gamma^{\leq N}(\tau,u)(t)\|_{L^2_x}^{2}
+
\int_0^T
\left(
\nu
\|\nabla\Gamma^{\leq N}u_{cf}\|_{L^2_x}^{2}
+
\mu
\|\nabla\Gamma^{\leq N}u_{df}\|_{L^2_x}^{2}
\right)dt
\\
&\lesssim
\|\Gamma^{\leq N}(\tau_0,u_0)\|_{L^2_x}^{2}
+
\sum_{|k|\leq N}
\mathcal E_N(T).
\end{aligned}
\end{equation}
By \eqref{def of E} and \eqref{def of D}, \eqref{eq:4.1} implies
\begin{equation}
\label{eq:4.2-energy}
\begin{aligned}
E_N(T)+D_N(T)
\lesssim
E_N(0)
+
\sum_{|k|\leq N}\mathcal E_N(T).
\end{aligned}
\end{equation}
Writing $\widetilde{\tau}=\Gamma^k\tau, \widetilde{u}=\Gamma^ku$, we set
\begin{equation}
\begin{aligned}
I_{1,k}(T)
&=
\int_0^T\int_{\mathbb R^3}
\left(|\nabla\tau|+| \nabla\cdot u|\right)
\left(|\widetilde{\tau}|^2+ |\widetilde{u}|^2\right) dxdt ,
\\
I_{2,k}(T)
&=
\int_0^T\int_{\mathbb R^3} \left[\mu |\nabla\widetilde u|
|\nabla\tau|
|\widetilde u| + (\lambda+\mu) |\nabla \cdot \widetilde u|
|\nabla\tau|
|\widetilde u| \right]dxdt,
\\
I_{3,k}(T)
&=
\int_0^T\int_{\mathbb R^3}
\left(
|\widetilde{\tau}F[\widetilde{\tau}]|
+
|\widetilde uF[\widetilde u]|
\right)dxdt,
\end{aligned}
\end{equation}
and then
\begin{equation}
\begin{aligned}
\mathcal E_N(T) \leq \sum\limits_{i=1}^3I_{i,k}(T).
\end{aligned}
\end{equation}
We estimate each nonlinear term $I_{i,k}(T)$.

\medskip
\noindent
{\bf Estimate of $I_{1,k}(T)$.} Using spherical coordinates, we have
\begin{equation}\label{412}
\begin{aligned}
I_{1,k}(T)
=&\int_0^T
\int_0^\infty
\int_{S^2}
\left(|\nabla\tau|+| \nabla\cdot u|\right)
\left(|\tilde{\tau}|^2+ |\tilde{u}|^2\right)
r^2d\omega dr dt\\
\lesssim& \int_0^T\int_0^\infty\left(|\nabla \tau|+|\nabla\cdot u|\right)_{L_\omega^\infty}\left(\|\widetilde{\tau}\|_{L^2_\omega}
+\|\widetilde{u}\|_{L^2_\omega}\right)^2r^2drdt\\
\lesssim&\int_0^T\int_0^\infty\left\|\nabla\Gamma^{\le 2}(\tau, u)\right\|_{L^2_\omega}\left(\|\widetilde{\tau}\|_{L^2_\omega}
+\|\widetilde{u}\|_{L^2_\omega}\right)^2r^2drdt\\
\lesssim&\int_0^T\sup_r\left\|\nabla\Gamma^{\le 2}(\tau, u)\right\|_{L^2_\omega}\left(\|\widetilde{\tau}\|_{L^2_x}
+\|\widetilde{u}\|_{L^2_x}\right)^2dt\\
\lesssim& \left\|(\tilde\tau,\tilde u)\right\|_{L^\infty(0,T;L^2_x)}
\left\|\Gamma^{\leq N_{*}}(\tau,u)\right\|_{L^2(0,T;L^\infty_rL^2_\omega)}\\
&\cdot \left\|\nabla\Gamma^{\leq N-1}(\tau,u)\right\|_{L^2(0,T; L^2_x)}.
\end{aligned}
\end{equation}

\medskip
\noindent
{\bf Estimate of $I_{2,k}(T)$.}
Similarly, using
\begin{equation}\label{4.5}
\left\|\nabla\tau\left(\mu\,\nabla\tilde u+(\lambda+\mu)\,\nabla\!\cdot\!\tilde u\right)\right\|_{L^2_\omega}
\lesssim
\|\Gamma^{\leq N_{*}}\tau\|_{L^2_\omega}\,
\|\nabla \Gamma^{\leq N}u\|_{L^2_\omega},
\end{equation}
we have
\begin{equation} \label{I2}
\begin{aligned}
I_{2,k}(T) \lesssim
\|\tilde u\|_{L^\infty(0,T;L^2_x)}
\|\Gamma^{\leq N_{*}}(\tau,u)\|_{L^2(0,T;L^\infty_rL^2_\omega)}
\|\nabla\Gamma^{\leq N}u\|_{L^2(0,T; L^2_x)}.
\end{aligned}
\end{equation}

\medskip
\noindent
{\bf Estimate of $I_{3,k}(T)$.} By \eqref{def F tau u}, the nonlinear term $F[\widetilde{\tau}]$ is a quadratic expression of the form
\begin{equation}
\Gamma^c (\tau,u)
\nabla \Gamma^{d}(\tau,u) ,
\qquad
|c|+|d|\leq N , \ |c| \geq 1 .
\end{equation}
By Lemma \ref{lem:L2omega-product}, for each fixed $r$, we have
\begin{equation} \label{estimate F tau}
\begin{aligned}
\|F[\tilde \tau]
\|_{L^2_\omega} \lesssim
\|
\Gamma^{\leq N_{*}}(\tau,u)
\|_{L^2_\omega}
\|
\nabla \Gamma^{\leq N-1}(\tau,u)
\|_{L^2_\omega}.
\end{aligned}
\end{equation}
Moreover, for the nonlinear term $F[\widetilde{u}]$, we have
\begin{equation} \label{estimate F u}
	\begin{aligned}
\|F[\tilde u]\|_{L^2_\omega}
& \lesssim
\left\|\Gamma^{\leq N_{*}}(\tau,u)\right\|_{L^2_\omega}\,
\left\|\nabla \Gamma^{\leq N-1}(\tau,u)\right\|_{L^2_\omega} \\
& \quad +
\Bigl(
\|\nabla \Gamma^{\le N-1}(\rho^{-1})\|_{L^2_\omega}
+
\|\tilde{\Omega}^{\le 4}(\rho^{-1})\|_{L^2_\omega}
\Bigr) \\
& \qquad \cdot
\bigl\|
\mu\,\nabla^2\Gamma^{\leq N_{*}}u,\,
(\lambda+\mu)\,\nabla^2\Gamma^{\leq N_{*}}u_{cf}
\bigr\|_{L^2_\omega}		\\
& \quad +
\|\Gamma^{\leq N_{*}}(\rho^{-1})\|_{L^2_\omega}\,
\bigl\|\mu\,\nabla \Gamma^{\le N}u,\,
(\lambda+\mu)\,\nabla  \cdot \Gamma^{\le N}u
\bigr\|_{L^2_\omega} \\
& \lesssim
\|\Gamma^{\leq N_{*}}(\tau,u)\|_{L^2_\omega}\,
\Bigl(
\|\nabla \Gamma^{\leq N-1} \tau\|_{L^2_\omega}
+
\|\nabla \Gamma^{\leq N}u\|_{L^2_\omega}
\Bigr).
\end{aligned}			
\end{equation}		
Hence,
\begin{equation} \label{I3}
\begin{aligned}
I_{3,k}(T)
&\lesssim
\|(\tilde\tau,\tilde u)\|_{L^\infty(0,T;L^2_x)}
\left(
\|F[\widetilde{\tau}]\|_{L^2(0,T; L^2_x)}
+
\|F[\widetilde u]\|_{L^2(0,T; L^2_x)}
\right)
\\
&\lesssim
\|(\tilde\tau,\tilde u)\|_{L^\infty(0,T;L^2_x)} \|\Gamma^{\leq N_{*}}(\tau,u)\|_{L^2(0,T;L^\infty_rL^2_\omega)}\\
& \qquad \cdot
\Bigl(
\|\nabla \Gamma^{\leq N-1} \tau\|_{L^2(0,T; L^2_x)}
+
\|\nabla \Gamma^{\leq N}u\|_{L^2(0,T; L^2_x)}
\Bigr).
\end{aligned}
\end{equation}

Combining \eqref{eq:4.2-energy}, \eqref{412}, \eqref{I2}, and \eqref{I3}, and then summing over
$|k|\leq N$, we obtain
\begin{equation} 
\begin{aligned}
& \quad E_N(T)+D_N(T) \\
& \le
E_N(0) +CE_N^{1/2}(T)\|\Gamma^{\leq N_{*}}(\tau,u)\|_{L^2(0,T;L^\infty_rL^2_\omega)} \\
& \qquad \cdot
\Bigl(
\|\nabla \Gamma^{\leq N-1} \tau\|_{L^2(0,T; L^2_x)}
+
\|\nabla \Gamma^{\leq N}u\|_{L^2(0,T; L^2_x)}
\Bigr),
\end{aligned}
\end{equation}
where $C$ is a sufficiently large constant independent of the bootstrap constant $M$. Since $\mu,\nu\geq\varepsilon$, the last factor in brackets is bounded by
\[\varepsilon^{-1/2}\bigl(  D_N^{1/2}(T)+  G_N^{1/2}(T) \bigr).
\]
Moreover, the trace factor is consistent with
$ \varepsilon^{1/2} A^{1/2}_{N_{*}}(T)$ by
\eqref{def of A low}. Hence, we obtain \eqref{tame high energy}.
Furthermore, by Lemma \ref{lem 2.2}, we have
\begin{equation} 
	\begin{aligned}
    & \quad
 \| \Gamma^{\le N_{*}}(\tau,u_{cf}) \|_{L^2(0,T;L^\infty_rL^2_\omega)} \\
 & \lesssim (\ln\varepsilon^{-1})^{\frac12}\,
\left\|\langle r\rangle^{-\frac12} \nabla \nabla^{\le 1 }\Gamma^{\le N_{*}}(\tau,u_{cf}) \right\|_{L^2(0,T; L^2_x)} \\
& \quad+ \varepsilon^{\frac12}\,\left\|\nabla \Gamma^{\le N_{*}}(\tau,u_{cf})\right\|_{L^2(0,T; L^2_x)}.
\end{aligned}
\end{equation}
It then follows from (BA1) and (BA3) that
\begin{equation} 
	\begin{aligned}
 \quad
 \| \Gamma^{\le N_{*}}(\tau,u_{cf}) \|_{L^2(0,T;L^\infty_rL^2_\omega)}   \le  2C M^3 \delta \varepsilon^{\frac12}.
\end{aligned}
\end{equation}
On the other hand, \eqref{ineq r f} and (BA2) give
\begin{equation}
	\begin{aligned}
    & \quad
\| \Gamma^{\le N_{*}} u_{df} \|_{L^2(0,T;L^\infty_rL^2_\omega)}  \le C \| \nabla \nabla^{\le 1 }\Gamma^{\le N_{*}} u_{df}\|_{L^2(0,T; L^2_x)}  \le C  M \delta \varepsilon^{\frac12},
\end{aligned}
\end{equation}
which implies
\begin{equation} \label{Gamma low Gamma hi}
	\begin{aligned}
 A^{\frac12}_{N_{*}}(T)  \le 4C M^3 \delta,
\end{aligned}
\end{equation}
and thus
\begin{equation} \label{Gamma low Gamma hi-2}
	\begin{aligned}
& \quad A^{\frac14}_{N_{*}}(T)
 \left(  E_N^{\frac12}(T)+  D_N^{\frac12}(T)+  G_N^{\frac12}(T)\right)  \\
& \le 4C^{\frac12} M^{\frac72} \delta^{\frac32} \varepsilon^{\frac12} (\ln\varepsilon^{-1})^{-1}.
\end{aligned}
\end{equation}
Substituting this estimate into \eqref{tame high energy} yields \eqref{EN DN}.

\end{proof}

\begin{lem}[density estimate]\label{lem:density}
Let $N\geq 10$. Under the assumptions of
Lemma \ref{lem: high-energy}, we have
\begin{equation} \label{tame density}
	\begin{aligned}
G_N (T) 
&\le C \left[E_N (T) + D_N (T) + D_N^{\frac{1}{2}} (T) G_N^{\frac{1}{2}} (T) \right]\\
&  \quad +C A_{N_{*}}^{\frac12} \left[E_N (T) + D_N (T) + G_N (T)\right],
\end{aligned}			
\end{equation}	
where $C>0$ is independent of $M, T,\mu$, and $\nu$. Moreover, under (BA1)-(BA3), we have
\begin{equation} \label{GN}
	\begin{aligned}
M^{-1} G_N^{\frac{1}{2}} (T)
& \le \left( C M^{-1} + 8C M^{\frac52}\delta^{\frac{1}{2}} + C^{\frac{1}{2}} M^{\frac{1}{2}} \right) \delta \varepsilon^{\frac{1}{2}} ( \ln \varepsilon^{-1} )^{-1}.
\end{aligned}			
\end{equation}	
\end{lem}

\begin{proof}
Summing \eqref{parabolic tau} over
$|k|\leq N-1$  gives
\begin{equation} \label{4.10}
	\begin{aligned}
G _N (T) &\le  \nu  \int_0^T \int_{\mathbb{R}^3} |\nabla \Gamma^{\le N-1}\tau |^2 dxdt \\
&\lesssim
\nu  \int_0^T \int_{\mathbb{R}^3}  |\nabla \Gamma^{\le N-1} u_{cf} |^2 dxdt \\
& \quad + \nu   \sup\limits_{t \in [0,T]} \|\nabla \Gamma^{\le N-1}u_{cf}  (t)\|_{L_x^2}   \| \Gamma^{\le N-1}\tau(t) \|_{L_x^2}   \\
& \quad + \nu   \sum\limits_{|k|\le N-1} \mathcal P_{k}(T) .
\end{aligned}			
\end{equation}	
The nonlinear contributions to $\mathcal P_{k}(T) $ are the $L^1((0,T)\times \mathbb R^3) $-norms of terms of the forms
\begin{equation}
\begin{aligned}
&(1)\ n\cdot \tilde{n}\cdot \tilde{\tau},\qquad
(2)\ \nabla \tau\cdot \nabla \tilde{\tau}\cdot \tilde{\tau}, \qquad (3)\ \nu\,\nabla \tau\cdot \tilde{\tau}\cdot \nabla \tilde{n}\\
&(4)\ \nu\,\nabla  \tilde{\tau}\cdot \nabla \tilde{n},\qquad
(5)\ \tilde{n}\,F[\tilde{\tau}],\qquad
(6)\ \tilde{\tau}\,F[\tilde{n}],
\end{aligned}
\end{equation}
for $|k| \le N-1$. We estimate these terms in turn.

\medskip

\noindent\textbf{Estimate of (1).} By H\"older's inequality and the bound $|\tilde{n}|\lesssim |\nabla \tilde{u}|$, for $|k|\le N-1$ we have
\begin{equation}
\begin{aligned}
&  \quad \|n\cdot \tilde{n}\cdot \tilde{\tau}\|_{L^1((0,T)\times \mathbb R^3) } \\
&\lesssim \|\tilde{\tau}\|_{L^\infty(0,T; L^2_x) } \| n \cdot \tilde{n}\|_{L^1(0,T; L^2_x) } \\
& \lesssim \|\tilde{\tau}\|_{L^\infty(0,T; L^2_x) } \| \Gamma^{\leq  N_{*}} u\|_{L^2(0,T; L^\infty_r L^2_\omega ) }  \|\nabla \Gamma^{\leq N-1}u \|_{L^2(0,T; L^2_x) }  \\
& \lesssim  A_{N_{*}}^{\frac12} (T) E_N^{\frac12} (T) D_N^{\frac12}(T) .
\end{aligned}
\end{equation}

\medskip

\noindent\textbf{Estimate of (2).} Similarly, for $|k|\le N-1$, we have
\begin{equation}
\begin{aligned}
&  \quad \|\nabla \tau\cdot \nabla \tilde{\tau}\cdot \tilde{\tau}\|_{L^1((0,T)\times \mathbb R^3) }\\
&\lesssim \|\tilde{\tau}\|_{L^\infty(0,T; L^2_x) } \| \nabla \tau\cdot \nabla \tilde{\tau} \|_{L^1(0,T; L^2_x) } \\
& \lesssim \|\tilde{\tau}\|_{L^\infty(0,T; L^2_x) } \| \Gamma^{\leq N_{*}} \tau\|_{L^2(0,T; L^\infty_r L^2_\omega ) }  \|\nabla \Gamma^{\leq N-1}\tau \|_{L^2(0,T; L^2_x) }  \\
& \lesssim  A_{N_{*}}^{\frac12}(T)  E_N^{\frac12} (T) G_N^{\frac12}(T) .
\end{aligned}
\end{equation}

\medskip

\noindent\textbf{Estimate of (3).}  We also have 
\begin{equation}
\begin{aligned}
&  \quad \nu \|\nabla \tau \cdot \tilde{\tau}\cdot \nabla\tilde{n}\|_{L^1((0,T)\times \mathbb R^3) } \\
&\lesssim \nu\|\tilde{\tau}\|_{L^\infty(0,T; L^2_x) } \| \nabla \tau \cdot \nabla\tilde{n}\|_{L^1(0,T; L^2_x) } \\
& \lesssim \nu\|\tilde{\tau}\|_{L^\infty(0,T; L^2_x) } \| \Gamma^{\leq  N_{*}} \tau \|_{L^2(0,T; L^\infty_r L^2_\omega ) }  \|\nabla \Gamma^{\leq N}  u \|_{L^2(0,T; L^2_x) }  \\
& \lesssim  A_{N_{*}}^{\frac12} (T) E_N^{\frac12} (T) D_N^{\frac12}(T) .
\end{aligned}
\end{equation}

\medskip

\noindent\textbf{Estimate of (4).} For $|k| \leq N-1$, we have
\begin{equation}
\begin{aligned}
&  \quad \nu \|\nabla  \tilde{\tau}\cdot \nabla\tilde{n}\|_{L^1((0,T)\times \mathbb R^3) } \\
& \lesssim \nu \| \nabla \Gamma^{\leq N-1}  \tau \|_{L^2(0,T; L^2_x) }  \|\nabla \Gamma^{\leq N}  u \|_{L^2(0,T; L^2_x) }  \\
& \lesssim  D_N^{\frac12} (T) G_N^{\frac12}(T) .
\end{aligned}
\end{equation}

\medskip

\noindent\textbf{Estimate of (5).} Recalling \eqref{estimate F tau}, we have for $|k| \leq N-1$
\begin{equation}
\begin{aligned}
& \quad \|\tilde{n}F[\tilde{\tau}]\|_{L^1((0,T)\times \mathbb R^3) } \\
& \lesssim \|\tilde{n}\|_{L^\infty(0,T; L^2_x) }  \|F[\tilde{\tau}]\|_{L^1(0,T; L^2_x) } \\
& \lesssim  \|\nabla \Gamma^{\leq N-1} u \|_{L^\infty(0,T; L^2_x) } \| \Gamma^{\leq N_{*}} (\tau,u) \|_{L^2(0,T; L^\infty_r L^2_\omega ) }  \|\nabla \Gamma^{\leq N-1}  (\tau,u) \|_{L^2(0,T; L^2_x) }  \\
& \lesssim  A_{N_{*}}^{\frac12} (T) E_N^{\frac12} (T) \left(D_N^{\frac12}(T) + G_N^{\frac12}(T)\right).
\end{aligned}
\end{equation}

\medskip

\noindent\textbf{Estimate of (6).} Recalling \eqref{def F n m} and \eqref{def F tau u}, we can obtain for $|k| \leq N-1$
\begin{equation}
\begin{aligned}
& \quad \|\tilde{\tau }F[\tilde{n}]\|_{L^1((0,T)\times \mathbb R^3) } \\
& \lesssim \|\tilde{\tau}\|_{L^\infty(0,T; L^2_x) }  \|F[\tilde{n}]\|_{L^1(0,T; L^2_x) } \\
& \lesssim  \| \Gamma^{\leq N-1} \tau \|_{L^\infty(0,T; L^2_x) } \| \Gamma^{\leq  N_{*}} (\tau,u) \|_{L^2(0,T; L^\infty_r L^2_\omega ) } \\
& \qquad \times \Bigl(
\|\nabla \Gamma^{\leq N-1} \tau\|_{L^2(0,T; L^2_x)}
+
\|\nabla \Gamma^{\leq N}u\|_{L^2(0,T; L^2_x)}
\Bigr) \\
& \lesssim  A_{N_{*}}^{\frac12} (T) E_N^{\frac12} (T) \left(D_N^{\frac12}(T) + G_N^{\frac12}(T)\right).
\end{aligned}
\end{equation}

\medskip

Combining \eqref{4.10} with the estimates of (1)-(6) gives \eqref{tame density}. Furthermore, substituting \eqref{Gamma low Gamma hi} into \eqref{tame density} and using (BA1), we obtain
\begin{equation} 
	\begin{aligned}
G_N^{\frac{1}{2}} (T)
& \le \left( C + 8CM^{\frac{7}{2}}\delta^{\frac{1}{2}} + C^{\frac{1}{2}} M^{\frac{3}{2}} \right) \delta \varepsilon^{\frac{1}{2}} ( \ln \varepsilon^{-1} )^{-1},
\end{aligned}			
\end{equation}	
which gives \eqref{GN}.

\end{proof}

Choose the bootstrap constant $M$ sufficiently large such that
\begin{equation} \label{take of M}
M\geq 8C^{\frac{3}{2}}
\end{equation}	
and $\delta$ sufficiently small but independent of $\varepsilon$; for example, take
\begin{equation} \label{take of delta}
\delta^{\frac{1}{2}} \le 2^{-5} C^{-\frac{1}{2}} M^{-\frac52},
\end{equation}	 
such that
\[ 
C^{\frac{1}{2}}
+
4 C^{\frac{1}{2}}M^{\frac72}\delta^{\frac{1}{2}} 
\leq \frac M4,
\]
and 
\[  C M^{-1} + 8C M^{\frac52}\delta^{\frac{1}{2}} + C^{\frac{1}{2}} M^{\frac{1}{2}} \leq \frac M4.
\]
Consequently, combining \eqref{EN DN} and \eqref{GN} yields
\[
E_N^{\frac{1}{2}}(T)+D_N^{\frac{1}{2}}(T) + M^{-1} G_N^{\frac{1}{2}}(T)
\leq
\frac M2
\delta\varepsilon^{\frac{1}{2}}
(\ln\varepsilon^{-1})^{-1},
\]
which improves the bootstrap assumption (BA1) and closes the continuity argument.

\medskip

\subsection{Parabolic estimate for $\Gamma^{\leq N_{\mathrm{lo}}} u_{df}$}

We now derive a low-order parabolic estimate for the divergence-free part of the velocity. Define the weighted energy
\begin{equation} \label{def of Mk}
	\begin{aligned}
\mathcal{M}_{k} :=  \left\|
\langle r\rangle^{-\frac12}\,\nabla (\Gamma^{\leq k}\tau,\Gamma^{\leq  k}u_{cf})
\right\|^2_{L^2(0,T; L^2_x)}.
\end{aligned}			
\end{equation}	
The resulting estimate is as follows.
\begin{lem}[low-order parabolic estimate]\label{lem:tame-low-df}
Let $N\geq 10$, and define
\[
P_{N_{\mathrm{lo}}}(T):=
 \|\Gamma^{\leq N_{\mathrm{lo}}}u_{df}\|^2_{L^\infty_tL^2_x}
 + \mu \|\nabla\Gamma^{\leq N_{\mathrm{lo}}}u_{df}\|^2_{L^2(0,T; L^2_x)}.
\]
Under the assumptions of
Lemma \ref{lem: high-energy}, we have
\begin{equation} \label{P BA1-3 1}
\begin{aligned}
&  \qquad P_{N_{\mathrm{lo}}}(T) \le P_{N_{\mathrm{lo}}}(0) \\
& + C\left[ P_{N_{\mathrm{lo}}}(T) A_{N_*}^{\frac 12}  (T) + P^{\frac 12}_{N_{\mathrm{lo}}}(T) D_N (T) + \varepsilon^{-\frac 12} P_{N_{\mathrm{lo}}}(T) \mathcal{M}^{\frac 12}_{N_{\mathrm{lo}}}(T) + \varepsilon^{-1}  P^{\frac 32}_{N_{\mathrm{lo}}}(T) \right] .
\end{aligned}
\end{equation}
where $C>0$ is independent of $M, T,\mu$, and $\nu$. Moreover, under (BA1)-(BA3), if $M$ is sufficiently large and $\delta$ is sufficiently small, then
\begin{equation} \label{P BA1-3 2}
\begin{aligned}
P^{\frac 12}_{N_{\mathrm{lo}}}(T) \le  \frac M2 \delta \varepsilon.
\end{aligned}
\end{equation}

\end{lem}

\begin{proof}
Let $\tilde m^{l} = \nabla \times  \Gamma^{\leq N_{\mathrm{lo}}} u_{df}$. By \eqref{eq tilde n m}, we obtain the evolution equation for $\tilde{m}^{l}$:
\begin{equation} \label{eq tilde m l}
\tilde{m}^{l}_t - \bar{\mu}\Delta \tilde{m}^{l} + (u\cdot\nabla)\tilde{m}^{l} = F[\tilde{m}^{l}],
\end{equation}
where the nonlinear source term is given by
\begin{equation}
\begin{aligned}
F[\tilde{m}^{l}] &= \nabla\times \left(\Gamma^{\leq N_{\mathrm{lo}}}u_{df}\right)_t - \nabla\times \left(\bar{\mu}\Delta \Gamma^{\leq N_{\mathrm{lo}}}u_{df}\right) + \nabla\bar{\mu}\Delta \Gamma^{\leq N_{\mathrm{lo}}}u_{df} \\
&\quad + \nabla\times \big((u\cdot\nabla)\Gamma^{\leq N_{\mathrm{lo}}}u_{df}\big) - \big((\nabla\times u)\cdot\nabla\big)\Gamma^{\leq N_{\mathrm{lo}}}u_{df}.
\end{aligned}
\end{equation}
Using \eqref{def F n m}, we have
\begin{equation} \label{def of F m l}
\begin{aligned}
F[\tilde m^l]
={}&
\sum_{\substack{|c|+|d|\leq N_{\mathrm{lo}}}}
\nabla\Gamma^c(\rho^{-1})
\times
\Big(
\nu\nabla(\nabla\cdot\Gamma^d u_{cf})
-\mu\nabla\times\nabla\times\Gamma^d u_{df}
\Big)
\\
&+\sum_{\substack{|c|+|d|\leq N_{\mathrm{lo}}}} [ (\nabla\times\Gamma^d u_{df}) \cdot \nabla  ] \Gamma^c u
-
\sum_{\substack{|c|+|d|\leq N_{\mathrm{lo}}}}
(\nabla\times\Gamma^d u_{df}) (\nabla
\cdot \Gamma^c u_{cf})
\\
&+
\sum_{\substack{|c|+|d|\leq N_{\mathrm{lo}}, |c| \geq 1}}
\Gamma^c(\rho^{-1})
\mu\Delta(\nabla\times\Gamma^d u_{df})
\\
&-
\sum_{\substack{|c|+|d|\leq N_{\mathrm{lo}}}}
(\Gamma^c u\cdot\nabla)
(\nabla\times\Gamma^d u_{df}).
\end{aligned}
\end{equation}
Using the vector identity
\[ \nabla \times (a \times b) = a \nabla \cdot b - b \nabla \cdot a + ( b \cdot \nabla ) a - ( a \cdot \nabla ) b, \]
we combine the second, third, and last sums of \eqref{def of F m l} through
\[
\begin{aligned}
&  \left[
(\nabla\times\Gamma^d u_{df})\cdot\nabla
\right]\Gamma^c u
-(\nabla\times\Gamma^d u_{df})
\operatorname{div}\Gamma^c u_{cf}
-(\Gamma^c u\cdot\nabla)
(\nabla\times\Gamma^d u_{df})\\
=&\nabla\times\left(
\Gamma^c u\times(\nabla\times\Gamma^d u_{df})
\right),\\
\end{aligned}
\]
Consequently, we may rewrite \eqref{def of F m l} as
\begin{equation} \label{def of F m 2}
\begin{aligned}
F[\tilde m^l]
={}&
\sum_{\substack{|c|+|d|\leq N_{\mathrm{lo}}}}
\nabla\Gamma^c(\rho^{-1})
\times
\Big(
\nu\nabla(\nabla\cdot\Gamma^d u_{cf})
-\mu\nabla\times\nabla\times\Gamma^d u_{df}
\Big)
\\
&+\sum_{\substack{|c|+|d|\leq N_{\mathrm{lo}}}} \nabla\times\left(
\Gamma^c u\times(\nabla\times\Gamma^d u_{df})
\right)
\\
&+
\sum_{\substack{|c|+|d|\leq N_{\mathrm{lo}}, |c| \geq 1}}
\Gamma^c(\rho^{-1})
\mu\Delta(\nabla\times\Gamma^d u_{df}).
\end{aligned}
\end{equation}

To obtain the energy estimate for $\tilde{m}$, we multiply \eqref{eq tilde m l} by $-\nabla\times \Gamma^{\leq N_{\mathrm{lo}}}u_{df}/|\nabla|^2$, which yields
\begin{equation}
\begin{aligned}
&\quad \frac12\frac{d}{dt}\|\Gamma^{\leq N_{\mathrm{lo}}}u_{df}\|_{L^2_x}^2 - \int_{\mathbb{R}^3} \bar\mu \Delta \tilde{m}^{l} \cdot |\nabla|^{-2} \tilde{m}^{l} dx \\
& \lesssim \int_{\mathbb{R}^3} (u \cdot \nabla ) \tilde{m}^{l} \cdot |\nabla|^{-2} \tilde{m}^{l} dx  + \int_{\mathbb{R}^3}  F[\tilde{m}^{l}]  \cdot |\nabla|^{-2} \tilde{m}^{l}  dx.
\end{aligned}
\end{equation}
Integrating in time over $[0,T]$,  we obtain
\begin{equation} \label{BA2 close 1}
\begin{aligned}
& \quad   \|\Gamma^{\leq N_{\mathrm{lo}}}u_{df}(T)\|_{ L^2_x}^2 + \mu \int_0^T\|\nabla \Gamma^{\leq N_{\mathrm{lo}}}u_{df}\|_{L^2_{x}}^2 dt \\
&\le \|\Gamma^{\leq N_{\mathrm{lo}}}u_{df}(0)\|_{L^2_x}^2 + C \sum\limits_{i=1}^4
\sum\limits_{j=1}^2 \mathcal{J}_{ij}(T),
\end{aligned}
\end{equation}
where 
\begin{equation}
\begin{aligned}
\mathcal{J}_{11}(T) & = \int_0^T \int_{\mathbb{R}^3} |
\Gamma^{\leq N_{\mathrm{lo}}}u_{df} | \cdot |\nabla \Gamma^{\leq N_{\mathrm{lo}}}u_{df} | \cdot | \nabla \Gamma^{\leq[N_{\mathrm{lo}}/2]+1} \tau + \Gamma^{\leq[N_{\mathrm{lo}}/2]+1}  u  | dxdt, \\
\mathcal{J}_{12}(T) & = \int_0^T \int_{\mathbb{R}^3} |
\Gamma^{\leq N_{\mathrm{lo}}}u_{df} | \cdot |\nabla \Gamma^{\leq[N_{\mathrm{lo}}/2]+1}  u_{df} | \cdot |  \nabla \Gamma^{\leq N_{\mathrm{lo}}} \tau + \Gamma^{\leq N_{\mathrm{lo}}} u  | dxdt, \\
\mathcal{J}_{21}(T) & = 
\nu \int_0^T \int_{\mathbb{R}^3} | \Gamma^{\leq N_{\mathrm{lo}}}u_{df} | \cdot  |\nabla \Gamma^{\leq[N_{\mathrm{lo}}/2]+1}  \tau |  \cdot |  \nabla \Gamma^{\leq N_{\mathrm{lo}}}u_{cf} | dxdt, \\
\mathcal{J}_{22}(T) & =  \nu \int_0^T \int_{\mathbb{R}^3} | \Gamma^{\leq N_{\mathrm{lo}}}u_{df} | \cdot  |\nabla \Gamma^{\leq N_{\mathrm{lo}}}  \tau |  \cdot |  \nabla \Gamma^{\leq[N_{\mathrm{lo}}/2]+1}u_{cf} | dxdt, \\
\mathcal{J}_{31}(T) & = \int_0^T \int_{\mathbb{R}^3}  |
|\nabla|^{-1}\Gamma^{\leq N_{\mathrm{lo}}}u_{df} |  \cdot | \nabla \Gamma^{\leq N_{\mathrm{lo}}}u_{df}  | \cdot  | \nabla^2  \Gamma^{\leq[N_{\mathrm{lo}}/2]+1} \tau + \nabla u_{cf}  | dxdt ,\\
\mathcal{J}_{32}(T) & =  \mu \int_0^T \int_{\mathbb{R}^3}  |
|\nabla|^{-1}\Gamma^{\leq N_{\mathrm{lo}}}u_{df} |  \cdot | \nabla \Gamma^{\leq[N_{\mathrm{lo}}/2]+1}u_{df}  | \cdot  | \nabla^2  \Gamma^{\leq N_{\mathrm{lo}}} \tau  | dxdt, \\
\mathcal{J}_{41}(T) & =  \nu  \int_0^T \int_{\mathbb{R}^3}  |
|\nabla|^{-1}\Gamma^{\leq N_{\mathrm{lo}}}u_{df} |  \cdot | \nabla \Gamma^{\leq N_{\mathrm{lo}}}u_{cf}  | \cdot  | \nabla^2  \Gamma^{\leq[N_{\mathrm{lo}}/2]+1} \tau | dxdt, \\
\mathcal{J}_{42}(T) & =  \nu \int_0^T \int_{\mathbb{R}^3}  |
|\nabla|^{-1}\Gamma^{\leq N_{\mathrm{lo}}}u_{df} |  \cdot | \nabla \Gamma^{\leq[N_{\mathrm{lo}}/2]+1}u_{cf}  | \cdot  | \nabla^2  \Gamma^{\leq N_{\mathrm{lo}}} \tau  | dxdt. \\
\end{aligned}
\end{equation}
We now estimate each nonlinear term $\mathcal{J}_{ij}(T) \ (1 \le i \le 4, \ 1 \le j \le 2)$.

\medskip

\noindent\textbf{Estimate of $\mathcal{J}_{11}(T)$.} We have
\begin{equation}
\begin{aligned}
\mathcal{J}_{11}(T) & \lesssim \left\|\Gamma^{\leq N_{\mathrm{lo}}}u_{df}\right\|_{L^\infty(0,T;L^2_x)} \cdot \varepsilon^{-\frac 12} 
\left\| \Gamma^{\leq[N_{\mathrm{lo}}/2]+3}(\nabla \tau,u)\right\|_{L^2(0,T;L^\infty_r L^2_\omega)}\\
& \qquad \cdot \varepsilon^{\frac 12}
\left\|\nabla \Gamma^{\leq N_{\mathrm{lo}}}u_{df}\right\|_{L^2(0,T; L^2_x)} \\
& \lesssim P_{N_{\mathrm{lo}}} (T)  A_{N_*}^{\frac 12} (T).
\end{aligned}
\end{equation}

\medskip
\noindent\textbf{Estimate of $\mathcal{J}_{12}(T)$.} We use \eqref{ineq r f} to get
\begin{equation}
\begin{aligned}
\mathcal{J}_{12}(T) & \lesssim \left\|\Gamma^{\leq N_{\mathrm{lo}}}u_{df}\right\|_{L^\infty(0,T;L^2_x)} \cdot \left\| \langle r \rangle^{\frac 12} \nabla  \Gamma^{\leq[N_{\mathrm{lo}}/2]+3} u_{df}\right\|_{L^2(0,T;L^\infty_r L^2_\omega)} \\
& \qquad \cdot \left(  \left\| \langle r \rangle^{-\frac 12} \nabla \Gamma^{ \le N_{\mathrm{lo}}} (\tau, u_{cf} )   \right\|_{L^2(0,T; L^2_x)} + \| \nabla \Gamma^{\le N_{\mathrm{lo}}} u_{df} \|_{L^2(0,T; L^2_x)} \right) \\
& \lesssim \varepsilon^{-1}  P^{\frac 32}_{N_{\mathrm{lo}}} (T)  + \varepsilon^{-\frac 12}  P_{N_{\mathrm{lo}}} (T) \mathcal{M}^{\frac 12}_{N_{\mathrm{lo}}} (T)  .
\end{aligned}
\end{equation}

\medskip

\noindent\textbf{Estimate of $\mathcal{J}_{21}(T)$.}
Using \eqref{sobolev}, we have
\begin{equation}
\begin{aligned}
\mathcal{J}_{21}(T)
& \lesssim \left\| \Gamma^{\leq N_{\mathrm{lo}}}u_{df} \right\|_{L^\infty(0,T;L^2_x)} \cdot \nu^{\frac 12} \left\| \nabla \Gamma^{ \leq[N_{\mathrm{lo}}/2]+3} \tau \right\|_{L^2(0,T; L^2_x)} \\
& \quad \cdot \nu^{\frac 12} \left\| \nabla \Gamma^{\leq N_{\mathrm{lo}}}u_{cf} \right\|_{L^2(0,T; L^2_x)} \\
& \lesssim P^{\frac 12}_{N_{\mathrm{lo}}} (T)  D_{N} (T).
\end{aligned}
\end{equation}

\medskip
\noindent\textbf{Estimate of $\mathcal{J}_{22}(T)$.}
Similarly, we have
\begin{equation}
\begin{aligned}
\mathcal{J}_{22}(T)
& \lesssim \left\| \Gamma^{\leq N_{\mathrm{lo}}}u_{df} \right\|_{L^\infty(0,T;L^2_x)} \cdot \nu^{\frac 12} \left\| \nabla \Gamma^{ \leq[N_{\mathrm{lo}}/2]+3} u_{cf} \right\|_{L^2(0,T; L^2_x)} \\
& \quad \cdot \nu^{\frac 12} \left\| \nabla \Gamma^{\leq N_{\mathrm{lo}}} \tau \right\|_{L^2(0,T; L^2_x)} \\
& \lesssim P^{\frac 12}_{N_{\mathrm{lo}}} (T)  D_{N} (T).
\end{aligned}
\end{equation}

\medskip

\noindent\textbf{Estimate of $\mathcal{J}_{31}(T)$.} By \eqref{ineq r f}, we have
\begin{equation}
\begin{aligned}
\mathcal{J}_{31}(T)
& \lesssim \left\| r^{\frac{1}{2}}
|\nabla|^{-1}\Gamma^{\leq N_{\mathrm{lo}}}u_{df}  \right\|_{L^\infty (0,T; L^\infty_r L^2_\omega) } \left\| \nabla \Gamma^{\leq N_{\mathrm{lo}}}u_{df} \right\|_{L^2(0,T; L^2_x)} \\
& \qquad \cdot\left\| r^{-\frac{1}{2}} ( \nabla^2  \Gamma^{\leq[N_{\mathrm{lo}}/2]+3} \tau + \nabla \Gamma^2 u_{cf} ) \right\|_{L^2(0,T; L^2_x)}   \\
& \lesssim \varepsilon^{-\frac{1}{2}}  P_{N_{\mathrm{lo}}} (T)  \mathcal{M}^{\frac 12}_{N_{\mathrm{lo}}} (T).
\end{aligned}
\end{equation}

\medskip

\noindent\textbf{Estimate of $\mathcal{J}_{32}(T)$.} 
By \eqref{ineq r f}, \eqref{hardy} and the interpolation inequality
\begin{equation} \label{interpolation}
\| f \|_{\dot{H}^{\frac{1}{2}}_x } \lesssim \| f \|^{\frac{1}{2}}_{L^2_x } \| \nabla f \|^{\frac{1}{2}}_{L^2_x },  \ \forall f \in \dot{H}^1 (\mathbb{R}^3) \cap L^2(\mathbb{R}^3),
\end{equation}
we have
\begin{equation}
\begin{aligned}
\mathcal{J}_{32}(T)
& \lesssim \left\| r^{\frac{1}{2}}
|\nabla|^{-1}\Gamma^{\leq N_{\mathrm{lo}}}u_{df}  \right\|_{L^\infty (0,T; L^\infty_r L^2_\omega) } \cdot \mu^{\frac 12} \left\| \nabla^2 \Gamma^{\leq N_{\mathrm{lo}}} \tau \right\|_{L^2(0,T; L^2_x)}   \\
& \qquad \cdot \mu^{\frac 12} \left\| r^{-\frac{1}{2}} \nabla  \Gamma^{\leq[N_{\mathrm{lo}}/2]+3} u_{df}   \right\|_{L^2(0,T; L^2_x)} \\
& \lesssim \left\|  \Gamma^{\leq N_{\mathrm{lo}}}u_{df}  \right\|_{L^\infty (0,T; L^\infty_r L^2_\omega) } \cdot \mu^{\frac 12} \left\| \nabla^2 \Gamma^{\leq N_{\mathrm{lo}}} \tau \right\|_{L^2(0,T; L^2_x)}   \\
& \qquad \cdot \mu^{\frac{1}{4}} \left\|\nabla  \Gamma^{\leq[N_{\mathrm{lo}}/2]+3} u_{df}   \right\|^{\frac 12}_{L^2(0,T; L^2_x)}  \cdot \mu^{\frac{1}{4}} \left\|\nabla^2  \Gamma^{\leq[N_{\mathrm{lo}}/2]+3} u_{df}   \right\|^{\frac 12}_{L^2(0,T; L^2_x)} \\
& \lesssim P_{N_{\mathrm{lo}}} (T)  D^{\frac 12}_N (T).
\end{aligned}
\end{equation}

\medskip

\noindent\textbf{Estimate of $\mathcal{J}_{41}(T)$.}
We also have
\begin{equation}
\begin{aligned}
\mathcal{J}_{41}(T)
& \lesssim \left\| r^{\frac{1}{2}}
|\nabla|^{-1}\Gamma^{\leq N_{\mathrm{lo}}}u_{df}  \right\|_{L^\infty (0,T; L^\infty_r L^2_\omega) }  \\
& \qquad \cdot  \nu^{\frac 12}  \left\| r^{-\frac{1}{2}}  \nabla^2 \Gamma^{\leq[N_{\mathrm{lo}}/2]+3} \tau \right\|_{L^2(0,T; L^2_x)} \cdot  \nu^{\frac 12}  \left\| \nabla \Gamma^{\leq N_{\mathrm{lo}}}u_{cf} \right\|_{L^2(0,T; L^2_x)} \\
& \lesssim \left\| \Gamma^{\leq N_{\mathrm{lo}}}u_{df}  \right\|_{L^\infty (0,T; L^\infty_r L^2_\omega) }   \cdot  \nu^{\frac 12}  \left\| \nabla \Gamma^{\leq N_{\mathrm{lo}}}u_{cf} \right\|_{L^2(0,T; L^2_x)} \\
& \qquad \cdot  \nu^{\frac 12}  \left\| \nabla^2 \Gamma^{\leq[N_{\mathrm{lo}}/2]+3} \tau \right\|^{\frac 12}_{L^2(0,T; L^2_x)} \left\| \nabla^3 \Gamma^{\leq[N_{\mathrm{lo}}/2]+3} \tau \right\|^{\frac 12}_{L^2(0,T; L^2_x)} \\
& \lesssim P^{\frac 12}_{N_{\mathrm{lo}}} (T)  D_N (T).
\end{aligned}
\end{equation}

\medskip

\noindent\textbf{Estimate of $\mathcal{J}_{42}(T)$.}
Similarly,
\begin{equation}
\begin{aligned}
\mathcal{J}_{42}(T)
& \lesssim \left\| r^{\frac{1}{2}}
|\nabla|^{-1}\Gamma^{\leq N_{\mathrm{lo}}}u_{df}  \right\|_{L^\infty (0,T; L^\infty_r L^2_\omega) }  \\
& \qquad \cdot  \nu^{\frac 12}  \left\| r^{-\frac{1}{2}}  \nabla^2 \Gamma^{\leq[N_{\mathrm{lo}}/2]+2} \tau \right\|_{L^2(0,T; L^2_x)} \cdot  \nu^{\frac 12}  \left\| \nabla \Gamma^{\leq[N_{\mathrm{lo}}/2]+1}u_{cf} \right\|_{L^2(0,T; L^2_x)} \\
& \lesssim P^{\frac 12}_{N_{\mathrm{lo}}} (T)  D_N (T).
\end{aligned}
\end{equation}

Combining \eqref{BA2 close 1} and the estimates of $\mathcal{J}_{ij}(T)$, we obtain \eqref{P BA1-3 1}.
Choosing $M$ and $\delta$ as in \eqref{take of M} and \eqref{take of delta}, respectively, gives \eqref{P BA1-3 2}. Consequently, the bootstrap assumption (BA2) is improved.

\end{proof}

\medskip

\subsection{Morawetz estimate for $\Gamma^{\leq N_{\mathrm{lo}}}\tau$ and $\Gamma^{\leq N_{\mathrm{lo}}}u_{cf}$}

We now derive the weighted spacetime estimate for the low-order compressible part of the solution. This is the final ingredient needed to improve (BA3), and the estimate ultimately follows from (BA1).

\begin{lem}[Morawetz estimate]
Let $N\geq 10$. Then
\begin{equation}\label{lem: Morawetz BA3 BA1}
	\begin{aligned}
\mathcal{M}^{\frac 12}_{N_{\mathrm{lo}}}  \le C_1
(\ln\varepsilon^{-1})^{\frac12} \left[ E_N^{\frac{1}{2}}(t) + D_N^{\frac{1}{2}}(t) + G_N^{\frac{1}{2}}(t)\right],
\end{aligned}			
\end{equation}	
where $\mathcal{M}_{k}$ is defined in \eqref{def of Mk}, and $C_1$ is independent of $M, T,\mu$, and $\nu$. Moreover, under (BA1)-(BA3), if $M$ is sufficiently large and $\delta$ is sufficiently small, then
\begin{equation} \label{M BA1-3 2}
\begin{aligned}
\mathcal{M}^{\frac{1}{2}}_{N_{\mathrm{lo}}} \le  \frac{M^3}{8} \delta \varepsilon^{\frac{1}{2}} (\ln\varepsilon^{-1})^{-\frac12} .
\end{aligned}
\end{equation}

\end{lem}

\begin{proof}
Note that (BA1) implies
\[
\|\Gamma^{\le5}(\tau,u)\|_{L^2_x}
\,|\ln\varepsilon^{-1}|
\ll1,~~~for~0<t<T.
\]
Therefore, for $|k| \le N_{\mathrm{lo}}$, it follows from the Morawetz estimate \eqref{eq:Morawetz-final} that
\begin{equation} \label{Morawetz EN}
\begin{aligned}
& \quad
\int_0^T \int_{ \mathbb R^3}
\left(
\frac{|\nabla\Gamma^{k }\tau|^2+|\nabla\Gamma^{k} u_{cf}|^2}
{\langle r \rangle}
+
\frac{|\Gamma^{k}\tau|^2}
{\langle r \rangle^2 r}
\right) dxdt
\\
&
\lesssim 
|\ln\varepsilon^{-1}|
\,\nu \int_0^T
\int_{ \mathbb R^3}
\left(
|\nabla \Gamma^{k}\tau|^2
+
|\Delta \Gamma^{k} n|^2
\right) dxdt \\
& \quad + |\ln\varepsilon^{-1}| \sup\limits_{t\in [0,T]}
\|\nabla \Gamma^{k} \tau\|_{L_x^2}
\|\nabla \Gamma^{k} u_{cf}\|_{L_x^2} \\
& \quad
+  \varepsilon \int_0^T
\int_{ \mathbb R^3}
\left(
|\nabla \Gamma^{k} \tau|^2
+
|\nabla \Gamma^{k} u_{cf}|^2
\right)dxdt \\
& \quad
+
|\ln\varepsilon^{-1}| \int_0^T \int_{ \mathbb R^3} \Big(\left|\tilde n\,(\partial_rF[\tilde\tau]+r^{-1}F[\tilde\tau])\right|
+
\left|(\tilde\tau_r+r^{-1}\tilde\tau)
F[\tilde n]\right|
\Big) dxdt \\
& \lesssim |\ln\varepsilon^{-1}| \left( E_N(T) + D_N(T) + G_N(T) + \mathcal N_{R }^{right} (T)\right) .
\end{aligned}
\end{equation}

For $\mathcal N_{R}^{right} (T)$ in \eqref{def of N right}, we use \eqref{def F tau u}, \eqref{def F n m}, Lemma \ref{lem:L2omega-product}, \eqref{hardy}, the bound $|\partial_r f | \le |\nabla f| $, and the Helmholtz estimate $\|\tilde n\|_{L^2_x}
\lesssim
\|\nabla\tilde u_{cf}\|_{L^2_x}$ to get
\begin{equation}
\begin{aligned}
\mathcal N_{R}^{right} (T)
& \lesssim \left( \|\Gamma^{N_{\mathrm{lo}}+1} u_{cf}  \|_{L^\infty(0,T;L^2_x)} +  \|\Gamma^{N_{\mathrm{lo}}+2}  \tau \|_{L^\infty(0,T;L^2_x)}   \right) \\
& \qquad \cdot \left\| \sum\limits_{\substack{|c|+|d| = N_{\mathrm{lo}} \\ |c| \geq 1}} \nabla^{\le 1}  \Gamma^{c} (\tau,u) \nabla   \nabla^{\le 1}  \Gamma^{d} (\tau,u) \right\|_{L_t^1(0,T; L_x^2)} \\
& \quad + \left\|\Gamma^{N_{\mathrm{lo}}+2}  \tau \right\|_{L^\infty(0,T;L^2_x)} \left\|  \Gamma^{\le  [N/4]+ 4} (\tau , u) \right\|_{L_t^2 (0,T;L_r^\infty L_\omega^2)}  \\
& \qquad \cdot \left(\| \nabla \Gamma^{\le  N_{\mathrm{lo}}} \tau \|_{L^2(0,T; L^2_x)} + \| \nabla \Gamma^{\le  N_{\mathrm{lo}}+1} u \|_{L^2(0,T; L^2_x)} \right) \\
& \lesssim A_{N_*}^{1/2} (T) \left( E_N(T) + D_N(T) + G_N(T)  \right).
\end{aligned}
\end{equation}
By \eqref{Gamma low Gamma hi}, \eqref{take of delta} and \eqref{take of M}, we have
\begin{equation} \label{estimate of N right}
\begin{aligned}
A_{N_*}^{1/2} (T) \ll 1.
\end{aligned}
\end{equation}
Therefore, by \eqref{Morawetz EN}, \eqref{estimate of N right} and (BA1), we obtain \eqref{lem: Morawetz BA3 BA1}. Furthermore, by (BA1), we have
\begin{equation}
	\begin{aligned}
\mathcal{M}^{\frac 12}_{N_{\mathrm{lo}}}  \le 2C_1 M^2 \delta \varepsilon^{\frac12}
(\ln\varepsilon^{-1})^{\frac12}.
\end{aligned}			
\end{equation}	
Finally, \eqref{M BA1-3 2} follows by taking $M=\max \{ 16C_1,  8C^{\frac{3}{2}}\}$ and choosing $\delta$ as in \eqref{take of delta}. Thus, each bound in (BA1)-(BA3) improves from $M$ to $M/2$, which closes the bootstrap argument.

\end{proof}

\section*{Acknowledgement}

N.-A Lai was partially supported by NSFC (No.12271487, W2521007). Y. Zhou was partially supported by NSFC (No. 12571231 and 12171097).

\bibliographystyle{plain}

\end{document}